\documentclass[10pt]{article}

\usepackage{amsfonts} 
\usepackage{amsmath,amssymb,amsthm,hhline,latexsym,mathrsfs}
\usepackage{relsize}
\usepackage{enumitem}

\DeclareMathAlphabet\mathbfcal{OMS}{cmsy}{b}{n}

\usepackage{hyperref}

\usepackage{color}
\usepackage{graphicx}
\usepackage{comment}

\newtheorem{definition}{Definition}[section]
\newtheorem{theorem}{Theorem}
\newtheorem{proposition}{Proposition}
\newtheorem{lemma}{Lemma}
\newtheorem{step}{Step}

\theoremstyle{definition}\newtheorem{remark}{Remark}

\newcommand{\R}{\mathbb{R}}
\newcommand{\N}{\mathbb{N}}
\newcommand{\A}{\mathcal{A}}

\renewcommand{\H}{\mathcal{H}}

\newcommand{\ph}{\varphi}

\renewcommand{\leq}{\leqslant}
\renewcommand{\geq}{\geqslant}

\begin{document}

\title{Constructive exact control of semilinear 1D wave equations by a least-squares approach}

\author{
\textsc{Arnaud M\"unch}\thanks{Laboratoire de Math\'ematiques Blaise Pascal, Universit\'e Clermont Auvergne, 
UMR CNRS 6620, Campus universitaire des C\'ezeaux, 3, place Vasarely, 63178, Aubi\`ere, France. E-mail: {\tt arnaud.munch@uca.fr.}}
\and
\textsc{Emmanuel Tr\'elat}\thanks{Sorbonne Universit\'e, CNRS, Universit\'e de Paris, Inria, Laboratoire Jacques-Louis Lions (LJLL), F-75005 Paris, France. E-mail: {\tt emmanuel.trelat@sorbonne-universite.fr.}}
}

\date{}
\maketitle

\begin{abstract}
It has been proved by Zuazua in the nineties that the internally controlled semilinear 1D wave equation $\partial_{tt}y-\partial_{xx}y + g(y)=f  1_{\omega}$, with Dirichlet boundary conditions, is exactly controllable in $H^1_0(0,1)\cap L^2(0,1)$ with controls $f\in L^2((0,1)\times(0,T))$, for any $T>0$ and any nonempty open subset $\omega$ of $(0,1)$, assuming that $g\in \mathcal{C}^1(\R)$ does not grow faster than $\beta\vert x\vert \ln^{2}\vert x\vert$ at infinity for some $\beta>0$ small enough.
The proof, based on the Leray-Schauder fixed point theorem, is however not constructive. 
In this article, we design a constructive proof and algorithm for the exact controllability of semilinear 1D wave equations.
Assuming that $g^\prime$ does not grow faster than $\beta \ln^{2}\vert x\vert$ at infinity for some $\beta>0$ small enough and that $g^\prime$ is uniformly H\"older continuous on $\R$ with exponent $s\in[0,1]$, we design a least-squares algorithm yielding an explicit sequence converging to a controlled solution for the semilinear equation, at least with order $1+s$ after a finite number of iterations.
\end{abstract}



\section{Introduction}
Let $\Omega:=(0,1)$, let $\omega:=(\ell_1,\ell_2)$ with $0\leq \ell_1<\ell_2\leq 1$ and let $T>0$. We set $Q_T:=\Omega\times (0,T)$, $q_T:=\omega\times (0,T)$ and $\Sigma_T:=\partial\Omega\times (0,T)$.
We consider the semilinear 1D wave equation
\begin{equation}\label{eq:wave-NL}
\boxed{
\left\{
\begin{aligned}
& \partial_{tt}y - \partial_{xx}y +  g(y)= f 1_{\omega} & \textrm{in}\  Q_T,\\
& y=0 & \textrm{on}\  \Sigma_T, \\
& (y(\cdot,0),\partial_t y(\cdot,0))=(u_0,u_1) & \textrm{in}\  \Omega,
\end{aligned}
\right.
}
\end{equation}
where $(u_0,u_1)\in \boldsymbol{V}:=H^1_0(\Omega)\times L^2(\Omega)$ is the initial state of $y$ and $f\in L^2(q_T)$ is a {\it control} function. Here and throughout the paper, $g:\R\to \R$ is a function of class $\mathcal{C}^1$ such that $\vert g(x)\vert\leq C (1+\vert x\vert)\ln^2(2+\vert x\vert)$ for every $x\in\R$, for some $C>0$.
Then, \eqref{eq:wave-NL} has a unique global (weak) solution in $\mathcal{C}^0([0,T]; H_0^1(\Omega)) \cap \mathcal{C}^1([0,T]; L^2(\Omega))$ (see \cite{CazenaveHaraux1980}).

We say that \eqref{eq:wave-NL} is exactly controllable in time $T$ if, for any $(u_0,u_1)\in \boldsymbol{V}$ and $(z_0,z_1)\in \boldsymbol{V}$, there exists a control function $f\in L^2(q_T)$ such that the solution of \eqref{eq:wave-NL} satisfies $(y(\cdot,T),\partial_t y(\cdot,T))=(z_0,z_1)$. 
The exact controllability problem for \eqref{eq:wave-NL} has been addressed in \cite{zuazua93}.

\begin{theorem}\label{thm_Zuazua}\cite{zuazua93} 
Assume that $T>2\max(\ell_1,1-\ell_2)$. There exists $\bar\beta>0$ (only depending on $\Omega$ and $T$) such that, if
\begin{equation}\label{asymptotic_behavior}
\limsup_{\vert x\vert \to +\infty} \frac{\vert g(x)\vert }{\vert x\vert \ln^2\vert x\vert} < \bar\beta
\end{equation}
then \eqref{eq:wave-NL} is exactly controllable in time $T$.
\end{theorem}

Moreover, it is proved in \cite{zuazua93} that, if $g$ behaves like $-s\ln^p(\vert s\vert)$ with $p>2$ as $\vert s\vert\rightarrow+\infty$, then the system is not exactly controllable in any time $T>0$, due to an uncontrollable blow-up phenomenon. 
Theorem \ref{thm_Zuazua} has been improved in \cite{Cannarsa_Loreti_Komornik2002}, weakening the condition \eqref{asymptotic_behavior} into 
$$
\limsup_{\vert x\vert\to +\infty} \ \Big\vert \int_0^x g(r)\, dr \Big\vert \ \Big(\vert x\vert \prod_{k=1}^{+\infty} \ln^{[k]}(e_k+x^2)\Big)^{-2}<+\infty
$$
where $\ln^{[k]}$ denotes the $k^{\textrm{th}}$ iterate of $\ln$ and $e_k>0$ is such that $\ln^{[k]}(e_k)=1$. This growth condition is essentially optimal since the solution of \eqref{eq:wave-NL} may blow up whenever $g$ grows faster at infinity and has the bad sign. 
The multi-dimensional case in which $\Omega$  is a bounded domain of $\R^d$, $d>1$, with a $\mathcal{C}^{1,1}$ boundary has been addressed in \cite{Li_Zhang_2000}. Assuming that the support $\omega$ of the control function is a neighborhood of $\partial\Omega$ and that $T>\textrm{diam}(\Omega\backslash\omega)$, the exact controllability of (\ref{eq:wave-NL}) is proved under the growth condition 
$\limsup_{\vert x\vert \to +\infty} \frac{\vert g(x)\vert }{\vert x\vert\ln^{1/2}\vert x\vert} < +\infty$. For control domains $\omega$ satisfying the classical multiplier assumption (see \cite{JLL88}), exact controllability has been proved in \cite{Zhang_2000} assuming that $g$ is globally Lipschitz continuous. We also mention \cite{coron-trelat-wave-06} where a positive boundary controllability result is proved for steady-state initial and final data and for $T$ large enough by a quasi-static deformation approach.

The proof given in \cite{zuazua93} is based on a Leray Schauder fixed point argument introduced in \cite{zuazua90,zuazua91} that reduces the exact controllability problem to obtaining suitable \textit{a priori} estimates for a linearized wave equation with a potential. More precisely, it is shown that the operator $K:L^\infty(Q_T)\to L^\infty(Q_T)$, where $y_\xi:=K(\xi)$ is a controlled solution with the control function $f_{\xi}$ of the linear boundary value problem 
\begin{equation}
\label{NL_z}
\left\{
\begin{aligned}
& \partial_{tt}y_{\xi} - \partial_{xx}y_{\xi} +  y_\xi \,\widehat{g}(\xi)= -g(0)+f_\xi 1_{\omega} &\textrm{in}\  Q_T,\\
& y_\xi=0 &\textrm{on}\  \Sigma_T, \\
& (y_\xi(\cdot,0),\partial_t y_{\xi}(\cdot,0))=(u_0,u_1) &\textrm{in}\  \Omega,
\end{aligned}
\right.
\qquad
\widehat{g}(x):=
\left\{ 
\begin{aligned}
& \frac{g(x)-g(0)}{x} & \textrm{if}\ x\neq 0\\
& g^{\prime}(0) & \textrm{if}\ x=0
\end{aligned}
\right.
\end{equation}
satisfying $(y_{\xi}(\cdot,T),\partial_t y_{\xi}(\cdot,T))=(z_0,z_1)$ has a fixed point. The control $f_{\xi}$ in \cite{zuazua93} is the one of minimal $L^2(q_T)$ norm. The existence of a fixed point for the operator $K$ is proved by applying the Leray-Schauder degree theorem: it is shown that if $\beta$ is small enough, then there exists $M=M(\Vert (u_0,u_1)\Vert_{\boldsymbol{V}}, \Vert (z_0,z_1)\Vert_{\boldsymbol{V}})>0$ such that $K$ maps the ball $B_\infty(0,M)$ to itself.  

The objective of this article is to design an algorithm providing an explicit sequence $(f_k)_{k\in \N}$ that converges strongly to an exact control for \eqref{eq:wave-NL}. A first idea that comes to mind is to consider the Picard iterations $(y_k)_{k\in\N}$ associated with the operator $K$ defined by $y_{k+1}=K(y_k)$, $k\geq 0$ initialized with any element $y_0\in L^{\infty}(Q_T)$. The resulting sequence of controls $(f_k)_{k\in \N}$ is then so that $f_{k+1}\in L^2(q_T)$ is the control of minimal $L^2(q_T)$ norm for $y_{k+1}$ solution of 
\begin{equation}
\label{NL_z_k}
\left\{
\begin{aligned}
& \partial_{tt}y_{k+1} - \partial_{xx}y_{k+1} +  y_{k+1} \,\widehat{g}(y_k)= -g(0)+f_{k+1} 1_{\omega} & \textrm{in}\quad Q_T,\\
& y_{k+1}=0 & \textrm{on}\  \Sigma_T, \\
& (y_{k+1}(\cdot,0),\partial_t y_{k+1}(\cdot,0))=(u_0,u_1) & \textrm{in}\  \Omega.
\end{aligned}
\right. 
\end{equation}
Such a strategy usually fails since the operator $K$ is in general not contracting, even if $g$ is globally Lipschitz. We refer to \cite{EFC-AM-2012} for numerical simulations providing evidence of the lack of convergence in parabolic cases (see also Remark \ref{rem_contract} in Appendix \ref{sec:linearizedwave}).
%
%

A second idea is to use a Newton type method in order to find a zero of the $\mathcal{C}^1$ mapping $\widetilde{F}: Y \to W$ defined by 
\begin{equation}\label{def-F}
\widetilde{F}(y,f):= \big(\partial_{tt}y - \partial_{xx}y  + g(y) - f 1_{\omega}, y(\cdot\,,0) - u_0, \partial_t y(\cdot\,,0) - u_1, y(\cdot\,,T) -z_0, \partial_t y(\cdot\,,T) - z_1\big) 
\end{equation}
for some appropriate Hilbert spaces $Y$ and $W$ (see further):
given $(y_0,f_0)$ in $Y$, the sequence $(y_k,f_k)_{k\in \N}$ is defined iteratively by $(y_{k+1},f_{k+1})=(y_k,f_k)-(Y_k,F_k)$ where $F_k$ is a control for $Y_k$ solution of  
\begin{equation}
\label{Newton-nn}
\left\{
\begin{aligned}
&\partial_{tt}Y_k-\partial_{xx}Y_k +  g^{\prime}(y_k)\,Y_{k} = F_{k} 1_{\omega} + \partial_{tt}y_k - \partial_{xx}y_k  + g(y_k) - f_k 1_{\omega},
                                                                                    & \quad  \textrm{in}\ Q_T,\\
& Y_k=0,                                                              & \quad  \textrm{on}\ \Sigma_T,\\ 
& Y_k(\cdot,0)=u_0-y_k(\cdot,0), \  \partial_t Y_k(\cdot,0)=u_1-\partial_t y_k(\cdot,0)                                   & \quad  \textrm{in}\ \Omega ,
\end{aligned}
\right.
   \end{equation}
   such that $Y_k(\cdot,T)=-y_k(\cdot,T)$ and $\partial_t Y_k(\cdot,T)=-\partial_t y_k(\cdot,T)$ in $\Omega$.
This linearization makes appear an operator $K_N$ such that $y_{k+1}=K_N(y_k)$ involving the first derivative of $g$. 
 However, as it is well known, such a sequence may fail to converge if the initial guess $(y_0,f_0)$ is not close enough to a zero of $F$ (see \cite{EFC-AM-2012} where divergence of the sequence is shown for large data).
  
%
The controllability of nonlinear partial differential equations has attracted a large number of works in the last decades (see \cite{Coron-Book-07} and references therein).
However, as far as we know, few are concerned with the approximation of exact controls for nonlinear partial differential equations, and the construction of convergent control approximations for controllable nonlinear equations remains a challenge. 

In this article, given any initial data $(u_0,u_1)\in \boldsymbol{V}$, we design an algorithm providing a sequence $(f_k)_{k\in \N}$ converging to a controlled solution for \eqref{eq:wave-NL}, under assumptions on $g$ that are slightly stronger than the one done in Theorem \ref{thm_Zuazua}. Moreover, after a finite number of iterations, the convergence is super-linear. This is done by introducing a quadratic functional measuring how much a pair $(y,f) \in Y$ is close to a controlled solution for \eqref{eq:wave-NL} and then by determining a particular minimizing sequence enjoying the announced property. A natural example of an error (or least-squares) functional is given by $\widetilde{E}(y,f):=\frac{1}{2}\Vert \widetilde{F}(y,f)\Vert^2_W$ to be minimized over $Y$. Exact controllability for~(\ref{eq:wave-NL}) is reflected by the fact that the global minimum of the nonnegative functional $\widetilde{E}$ is zero, over all pairs $(y,f)\in Y$ solutions of \eqref{eq:wave-NL}. In the line of recent works on the Navier-Stokes system (see \cite{lemoine-Munch-Pedregal-AMO-20}), we determine, using an appropriate descent direction, a minimizing sequence $(y_k,f_k)_{k\geq 0}$ converging to a zero of the quadratic functional.
 
The paper is organized as follows.  In Section \ref{sec:LS}, we define the (nonconvex) least-squares functional $E$ and the corresponding (nonconvex) optimization problem \eqref{extremal_problem}. We show that $E$ is Gateaux-differentiable on $\mathcal{A}$ and that any critical point $(y,f)$ for $E$ such that $g^\prime(y)\in L^\infty(Q_T)$ is also a zero of $E$. This is done by introducing an adequate descent direction $(Y^1,F^1)$ for $E$ at any $(y,f)$ for which $E^\prime(y,f)\cdot (Y^1,F^1)$ is proportional to $\sqrt{E(y,f)}$. This instrumental fact compensates the failure of convexity of $E$ and is at the base of the global convergence properties of our least-squares algorithm. The design of this algorithm is done by determining a minimizing sequence based on $(Y^1,F^1)$, which is proved to converge to a controlled pair for the semilinear wave equation \eqref{eq:wave-NL}, in our main result (Theorem \ref{mainthm}), under appropriate assumptions on $g$. Moreover, we prove that, after a finite number of iterations, the convergence is super-linear.
Theorem \ref{mainthm} is proved in Section \ref{sec_proof_mainthm}.
We show in Section \ref{sec:remarks} that our least-squares approach coincides with the classical \emph{damped Newton method} applied to a mapping similar to $\widetilde{F}$, and we give a number of other comments. 
In Appendix \ref{sec:linearizedwave}, we state some \textit{a priori} estimates for the linearized wave equation with potential in $L^\infty(Q_T)$ and source term in $L^2(Q_T)$ and we show that the operator $K$ is contracting if $\Vert \hat{g}^\prime\Vert_{L_{\textrm{loc}}^\infty(\R)}$ is small enough.

As far as we know, the method introduced and analyzed in this work is the first one providing an explicit, algorithmic construction of exact controls for semilinear wave equations. 

\paragraph{Notations.}
Throughout, we denote by $\Vert \cdot \Vert_{\infty} $ the usual norm in $L^\infty(\R)$, by $(\cdot,\cdot)_X$ the scalar product of $X$ (if $X$ is a Hilbert space) and by $\langle \cdot, \cdot \rangle_{X,Y}$ the duality product between $X$ and $Y$.
The notation $\Vert \cdot\Vert_{2,q_T}$ stands for $\Vert \cdot\Vert_{L^2(q_T)}$ and $\Vert \cdot\Vert_p$ for $\Vert \cdot\Vert_{L^p(Q_T)}$, mainly for $p=2$ and $p=+\infty$.

Given any $s\in (0,1]$, we denote by $\mathcal{C}^{1,s}(\R)$ the set of all functions $g\in \mathcal{C}^1(\R)$ such that $g^\prime$ is uniformly H\"older continous with exponent $s$, meaning that
$$
[g']_s := \sup_{a,b\in \R \atop a\neq b} \frac{\vert g^\prime(a)-g^\prime(b)\vert}{\vert a-b\vert^s}  < +\infty .
$$
For $s=0$, by extension, we set $[g']_0 :=2\Vert g^\prime\Vert_\infty$.
In particular, $g\in \mathcal{C}^{1,0}(\R)$ if and only if $g\in \mathcal{C}^1(\R)$ and $g^\prime\in L^\infty(\R)$, and $g\in \mathcal{C}^{1,1}(\R)$ if and only if $g^\prime$ is Lipschitz continuous (in this case, $g^\prime$ is almost everywhere differentiable and $g^{\prime\prime}\in L^\infty(\R)$, and we have $[g^\prime]_s\leq \Vert g^{\prime\prime}\Vert_\infty$).

\section{Least-squares algorithm and main result}\label{sec:LS}

\subsection{Least-squares functional and minimization problem}\label{sec_lsq_method}
\paragraph{Least-squares functional.}
We consider the Hilbert space
\begin{multline*}
\H :=
\Big\{(y,f)\in L^2(Q_T)\times L^2(q_T) \ \mid\   \partial_{tt}y - \partial_{xx}y\in L^2(Q_T),  \ y=0 \ \textrm{on}\ \Sigma_T ,  \\
(y(\cdot,0),\partial_t y(\cdot,0)) \in \boldsymbol{V} , \ (y(\cdot,T),\partial_t y(\cdot,T))\in \boldsymbol{V} \Big\}
\end{multline*}
endowed with the scalar product
\begin{multline*}
((y_1,f_1),(y_2,f_2))_{\H} := (y_1,y_2)_2+ \big( (y_1(\cdot,0),\partial_t y_1(\cdot,0)),(y_2(\cdot,0),\partial_ty_2(\cdot,0)) \big)_{\boldsymbol{V}} \\
+ \big(\partial_{tt}y_1 - \partial_{xx}y_1,\partial_{tt}y_2-\partial_{xx}y_2 \big)_2+ (f_1,f_2)_{2,q_T} 
\end{multline*}
and the norm $\Vert (y,f)\Vert_{\H}:=\sqrt{((y,f),(y,f))_{\H}}$. 

In what follows, we fix some arbitrary $(u_0,u_1)\in \boldsymbol{V}$ and $(z_0,z_1)\in \boldsymbol{V}$.
The subspaces of $\H$ defined by
\begin{equation*}
\begin{split}
\A  & := \Big\{(y,f)\in \H \ \mid\  
(y(\cdot,0),\partial_t y(\cdot,0))=(u_0,u_1) , \ (y(\cdot,T),\partial_t y(\cdot,T))=(z_0,z_1)\ \textrm{in}\ \Omega \Big\} , \\
\A_0 & := \Big\{(y,f)\in \H \ \mid\ (y(\cdot,0),\partial_t y(\cdot,0))=(0,0),\ (y(\cdot,T),\partial_t y(\cdot,T))=(0,0)\ \textrm{in}\ \Omega \Big\}  ,
\end{split}
\end{equation*}
Note that $\mathcal{A}=(\overline{y},\overline{f})+\mathcal{A}_0$ for any $(\overline{y},\overline{f})\in \mathcal{A}$.

Given any $(y,f)\in\mathcal{A}$, it follows from the \textit{a priori} estimate for the linear 1D wave equation that there exists $C>0$, only depending on $\Omega$ and $T$, such that
\begin{equation}\label{estiminftyA}
\begin{split}
\Vert (y,\partial_t y)\Vert^2_{L^{\infty}(0,T; \boldsymbol{V})} &\leq C \Big( \Vert \partial_{tt}y-\partial_{xx}y\Vert^2_{L^2(Q_T)}+\Vert (u_0,u_1) \Vert^2_{\boldsymbol{V}}\Big) \\
\Vert y\Vert_\infty &\leq C\Vert (y,f)\Vert_\H
\end{split}
\end{equation}
in particular $y\in L^\infty(Q_T)$. 
Since $g$ is of class $\mathcal{C}^1$, we have $g(y)\in L^2(Q_T)$ and $g^\prime(y)\in L^\infty(Q_T)$.
We define the least-squares functional $E:\mathcal{A}\to \R$ by
$$
\boxed{
E(y,f):=\frac{1}{2}\big\Vert \partial_{tt}y-\partial_{xx}y + g(y)-f 1_{\omega}\big\Vert^2_{L^2(Q_T)} 
}
$$
for every $(y,f)\in\mathcal{A}$.

\paragraph{Least-squares minimization problem.}
For any fixed $(\overline{y},\overline{f})\in \mathcal{A}$, we consider the (nonconvex) minimization problem
\begin{equation}\label{extremal_problem}
\boxed{
\inf_{(y,f)\in \mathcal{A}_0} E(\overline{y}+y,f+\overline{f})
}
\end{equation}

In the framework of Theorem \ref{thm_Zuazua}, the infimum of the functional of $E$ is zero and is reached by at least one pair $(y,f)\in \mathcal{A}$, solution of  \eqref{eq:wave-NL} and satisfying $(y(\cdot,T),\partial_t y(\cdot,T))=(z_0,z_1)$. Conversely, any pair $(y,f)\in \mathcal{A}$ such that $E(y,f)=0$ is a solution of \eqref{eq:wave-NL}. In this sense, the functional  $E$ is an \emph{error functional} which measures the deviation of $(y,f)$ from being a solution of the underlying nonlinear equation.  

A classical algorithmic way for computing the minimum consists in following descent directions, along the gradient of the functional. In descent algorithms, local minima are a usual issue to face with, unless the functional $E$ is convex. Since \eqref{eq:wave-NL} is nonlinear, here $E$ fails to be convex in general. In spite of that, we are going to construct a minimizing sequence which always converges to a zero of $E$. 

\begin{definition}
Let $T>2\max(\ell_1,1-\ell_2)$ be arbitrary. 
Given any $(y,f)\in \mathcal{A}$, over all pairs $(Y^1,F^1)\in \mathcal{A}_0$ solutions (the next result shows that there do exist some solutions) of  
\begin{equation}\label{wave-Y1}
\left\{
\begin{aligned}
& \partial_{tt}Y^1 - \partial_{xx}Y^1 +  g^{\prime}(y)\cdot Y^1 = F^1 1_{\omega}+\big(\partial_{tt}y-\partial_{xx}y + g(y)-f 1_{\omega}\big) &\textrm{in}\  Q_T,\\
& Y^1=0 &\textrm{on}\  \Sigma_T, \\
& (Y^1(\cdot,0),\partial_t Y^1(\cdot,0))=(0,0) & \textrm{in}\  \Omega,
\end{aligned}
\right.
\end{equation}
we select the (unique) pair $(Y^1,F^1)\in \mathcal{A}_0$ such that the control $F^1$, which is a null control for $Y^1$, has a minimal $L^2(q_T)$ norm. In what follows, it is called \emph{the solution $(Y^1,F^1)\in \mathcal{A}_0$ of \eqref{wave-Y1} of minimal control norm}.
\end{definition}


In the result hereafter, given any $(y,f)\in\mathcal{A}$, we establish some properties of the pair $((y,f),(Y^1,F^1))\in\A\times\A_0$, where $(Y^1,F^1)$ is the solution of \eqref{wave-Y1} of minimal control norm, which are at the base of the least-squares algorithm that we propose in Section \ref{sec_algo}, and are useful in view of proving its convergence (see Theorem \ref{mainthm}).

\begin{proposition}\label{propo1}
Assume that $T>2\max(\ell_1,1-\ell_2)$. There exists a positive constant $C$, only depending on $\Omega$ and $T$, such that, given any $(y,f)\in\mathcal{A}$:
\begin{enumerate}[label=(\roman*)]
\item\label{estimateC1C2} There exist solutions of \eqref{wave-Y1}. Moreover, the solution $(Y^1,F^1)\in \mathcal{A}_0$ of \eqref{wave-Y1} of minimal control norm is unique and satisfies
\begin{equation} \label{estimateF1Y1}
\Vert (Y^1,\partial_t Y^1)\Vert_{L^{\infty}(0,T;\boldsymbol{V})}+ \Vert F^1\Vert_{2,q_T} \leq C  e^{C\sqrt{\Vert g^{\prime}(y)\Vert_{\infty}} }\sqrt{E(y,f)}
\end{equation}
and 
\begin{equation} \label{estimateF1Y1_A0}
\Vert (Y^1,F^1)\Vert_\H \leq C  e^{C\sqrt{\Vert g^{\prime}(y)\Vert_{\infty}} }\sqrt{E(y,f)} .
\end{equation}
In particular, $\Vert Y^1\Vert_{L^{\infty}(Q_T)} \leq C  e^{C\sqrt{\Vert g^{\prime}(y)\Vert_{\infty}} }\sqrt{E(y,f)}$.

\item\label{differentiabiliteE} The derivative of $E$ at $(y,f)\in \mathcal{A}$ along the direction $(Y^1,F^1)$ satisfies 
\begin{equation}\label{estimateEEprime}
E^{\prime}(y,f)\cdot (Y^1,F^1):= \lim_{\lambda\to 0 \atop \lambda\neq 0} \frac{E((y,f)+\lambda (Y^1,F^1))-E(y,f)}{\lambda} = 2E(y,f).
\end{equation}

\item\label{propo1_item3} Noting that the derivative $E^{\prime}(y,f)$ does not depend on $(Y,F)$ and defining the norm $\Vert E^{\prime}(y,f)\Vert_{\mathcal{A}_0^{\prime}}:= \displaystyle\sup_{(Y,F)\in \mathcal{A}_0\setminus\{0\}} \frac{E^{\prime}(y,f)\cdot (Y,F)}{\Vert (Y,F)\Vert_\H}$, where $\mathcal{A}_0^{\prime}$ is the topological dual of $\mathcal{A}_0$, we have
\begin{equation}\label{ineq_E_Eprime}
\frac{1}{\sqrt{2}\max\big(1,\Vert g^\prime(y)\Vert_\infty\big)}\Vert E^{\prime}(y,f)\Vert_{\mathcal{A}_0^\prime}\leq \sqrt{E(y,f)}\leq \frac{1}{\sqrt{2}}C e^{C\sqrt{\Vert g^{\prime}(y)\Vert_\infty}} \Vert E^{\prime}(y,f)\Vert_{\mathcal{A}_0^\prime} .
\end{equation}

\item\label{propo1_estimateE} Assume that $g\in \mathcal{C}^{1,s}(\R)$ for some $s\in [0,1]$. 
Then
\begin{equation}\label{estimateE}
E\big((y,f)-\lambda (Y^1,F^1)\big)\leq \Big(\vert 1-\lambda\vert +\lambda^{1+s}\,K(y) E(y,f)^{\frac{s}{2}}\Big)^2 \, E(y,f)   \qquad \forall \lambda\in \R
\end{equation}
where
\begin{equation}\label{estimateE_C}
K(y):= C\, [g^\prime]_s  \Big(C  e^{C\sqrt{\Vert g^{\prime}(y)\Vert_\infty}}\Big)^{1+s} .
\end{equation}
\end{enumerate}
\end{proposition}

\begin{proof} 
Let us establish \ref{estimateC1C2}.
The first estimate is a consequence of Lemma \ref{controllability_result} in Appendix \ref{sec:linearizedwave}, using the equality $\Vert \partial_{tt} y-\Delta y +g(y)-f 1_{\omega}\Vert_2=\sqrt{2E(y,f)}$. The second one follows from 
\begin{equation}
\nonumber
\begin{aligned}
\Vert (Y^1,F^1)\Vert_\H & \leq  \Vert \partial_{tt}Y^1-\partial_{xx}Y^1\Vert_2 + \Vert Y^1\Vert_2+\Vert F^1\Vert_{2,q_T}+\Vert Y^1(\cdot,0),\partial_t Y^1(\cdot,0)\Vert_{\boldsymbol{V}} \\
& \leq (1+\Vert g^{\prime}(y)\Vert_\infty)\Vert Y^1\Vert_2 + 2\Vert F^1\Vert_{2,q_T}+\sqrt{E(y,f)}\\
& \leq C (1+\Vert g^{\prime}(y)\Vert_\infty)e^{C\sqrt{\Vert g^{\prime}(y)\Vert_{\infty}}}\sqrt{E(y,f)}\\
& \leq C e^{(2+C)\sqrt{\Vert g^{\prime}(y)\Vert_{\infty}}}\sqrt{E(y,f)}
\end{aligned}
\end{equation}
using that $(1+s)\leq e^{2\sqrt{s}}$ for every $s\geq 0$.

To prove \ref{differentiabiliteE}, we first check that, for every $(Y,F)\in \mathcal{A}_0$, the functional $E$ is differentiable at $(y,f)\in \mathcal{A}$ along the direction $(Y,F)\in \mathcal{A}_0$. For any $\lambda\in \R$, simple computations lead to
$$
E(y+\lambda Y,f+\lambda F) =E(y,f)+ \lambda E^{\prime}(y,f)\cdot (Y,F) + h((y,f),\lambda (Y,F))
$$
with 
\begin{equation}\label{Efirst}
E^{\prime}(y,f)\cdot (Y,F):=\big(\partial_{tt}y-\partial_{xx}y+ g(y)-f 1_\omega,  \partial_{tt}Y-\partial_{xx}Y+ g^\prime(y)Y-F 1_\omega\big)_2
\end{equation}
and
\begin{multline*}
h((y,f),\lambda (Y,F)):= \frac{\lambda^2}{2}\big(\partial_{tt}Y-\partial_{xx}Y+ g^\prime(y)Y-F 1_\omega,\partial_{tt}Y-\partial_{xx}Y+ g^\prime(y)Y-F 1_\omega\big)_2 \\
+ \lambda \big(\partial_{tt}Y-\partial_{xx}Y+ g^\prime(y)Y-F 1_\omega,\ell(y,\lambda Y)\big)_2  \\
+\big(\partial_{tt}y-\partial_{xx}y+ g(y)-f 1_\omega,\ell(y,\lambda Y)\big)_2+ \frac{1}{2}(\ell(y,\lambda Y),\ell(y,\lambda Y))
\end{multline*}
where 
\begin{equation}\label{defell}
\ell(y,\lambda Y):=g(y+\lambda Y)-g(y)-\lambda g^{\prime}(y)Y .
\end{equation}
The mapping $(Y,F)\mapsto E^{\prime}(y,f)\cdot (Y,F)$ is linear continuous from $\mathcal{A}_0$ to $\R$ since
\begin{equation}
\label{useful_estimate}
\begin{aligned}
\vert E^{\prime}(y,f)\cdot (Y,F)\vert &\leq  \Vert \partial_{tt}y-\partial_{xx}y+ g(y)-f 1_\omega\Vert_2 \Vert \partial_{tt}Y-\partial_{xx}Y+ g^\prime(y)Y-F 1_\omega\Vert_2\\
& \leq \sqrt{2E(y,f)} \, \big(\Vert (\partial_{tt}Y-\partial_{xx}Y)\Vert_2 + \Vert g^\prime(y)\Vert_{\infty}\Vert Y\Vert_2 + \Vert F\Vert_{2,q_T}\big)\\
& \leq \sqrt{2E(y,f)}\, \max\big(1,\Vert g^\prime(y)\Vert_\infty\big) \Vert (Y,F)\Vert_\H.
\end{aligned}
\end{equation}
Similarly, for every $\lambda\in \R\setminus\{0\}$, 
\begin{multline*}
\Big\vert \frac{1}{\lambda} h((y,f),\lambda (Y,F))\Big\vert  \leq \frac{\vert\lambda\vert}{2} \Vert \partial_{tt}Y-\partial_{xx}Y+ g^\prime(y)Y-F 1_\omega\Vert^2_2 \\
+\Big( \vert\lambda\vert \Vert \partial_{tt}Y-\partial_{xx}Y+ g^\prime(y)Y-F 1_\omega\Vert_2 +\sqrt{2E(y,f)}+\frac{1}{2}\Vert \ell(y,\lambda Y)\Vert_2\Big) \frac{1}{\vert\lambda\vert}\Vert \ell(y,\lambda Y)\Vert_2.
\end{multline*}
Since $g^\prime\in L_{\mathrm{loc}}^\infty(\R)$ and $y\in L^{\infty}(Q_T)$, we have
$$
\Big\vert \frac{1}{\lambda}\ell(y,\lambda Y)\Big\vert = \Big\vert \frac{g(y+\lambda Y)-g(y)}{\lambda}-g^{\prime}(y)Y\Big\vert\leq \big(\sup_{\theta\in (0,1)} \Vert g^{\prime}(y+\theta Y)\Vert_{\infty}+\Vert g^{\prime}(y)\Vert_{\infty}\big) \vert Y\vert 
$$
a.e. in $Q_T$,
and $\big\vert \frac{1}{\lambda}\ell(y,\lambda Y)\big\vert = \big\vert \frac{g(y+\lambda Y)-g(y)}{\lambda}-g^{\prime}(y)Y\big\vert\to 0$ as $\lambda\to 0$ a.e. in $Q_T$. By the Lebesgue dominated convergence theorem, it follows that $\vert \frac{1}{\lambda}\vert \Vert \ell(y,\lambda Y)\Vert_2\to 0$ as $\lambda\to 0$ and then that $\vert h((y,f),\lambda (Y,F))\vert=o(\lambda)$. 
We deduce that the functional $E$ is differentiable at the point $(y,f)\in \mathcal{A}$ along the direction $(Y,F)\in \mathcal{A}_0$.
Finally, \eqref{estimateEEprime} follows from the definition of $(Y^1,F^1)$ given in (\ref{wave-Y1}). 

Let us establish \ref{propo1_item3}.
Note that, by \eqref{Efirst}, the derivative $E^{\prime}(y,f)$ does not depend on $(Y,F)$. 
Now, \eqref{estimateEEprime} gives $E(y,f)=\frac12 E^{\prime}(y,f)\cdot (Y^1,F^1)$
where $(Y^1,F^1)\in \mathcal{A}_0$ is solution of (\ref{wave-Y1}) and, using \eqref{estimateF1Y1_A0},
$$E(y,f)\leq\frac12 \Vert E'(y,f)\Vert_{ \mathcal{A}_0'} \Vert(Y^1,F^1)\Vert_{ \mathcal{A}_0}\leq \frac12  C e^{C\sqrt{\Vert g^{\prime}(y)\Vert_\infty}} \Vert E'(y,f)\Vert_{ \mathcal{A}_0'}\sqrt{E(y,f)}.
$$
Besides, for all $(Y,F)\in\mathcal{A}_0$, the inequality $\vert E^{\prime}(y,f)\cdot (Y,F)\vert \leq  \sqrt{2E(y,f)} \max(1,\Vert  g^\prime(y)\Vert_{\infty}) \Vert (Y,F)\Vert_\H$ coming from \eqref{useful_estimate} leads to the left inequality in \eqref{ineq_E_Eprime}.

Let us finally establish \ref{propo1_estimateE}. We start by observing that, since $g(y-\lambda Y^1) = \ell(y,-\lambda Y^1)+g(y)-\lambda g^\prime(y) Y^1$, and since $(Y^1,F^1)\in \mathcal{A}_0$ is solution of \eqref{wave-Y1}, we have
\begin{equation}\label{Eell}
\begin{aligned}
& E\big((y,f)-\lambda (Y^1,F^1)\big) \\
&\quad= \frac{1}{2} \big\Vert \big(\partial_{tt}y-\partial_{xx}y+g(y)-f1_\omega\big)-\lambda \big(\partial_{tt}Y^1-\partial_{xx}Y^1+g^\prime(y)Y^1-F^1 1_\omega\big)+\ell(y,-\lambda Y^1)\big\Vert_2^2 \\
&\quad = \frac{1}{2} \big\Vert (1-\lambda)\big(\partial_{tt}y-\partial_{xx}y+g(y)-f1_\omega\big)+\ell(y,-\lambda Y^1)\big\Vert_2^2 .
\end{aligned}
\end{equation}
Now, for any $(u,v)\in \R^2$ and any $\lambda\geq 0$, writing $g(u+\lambda v)-g(u)=v\int_0^\lambda g^\prime(u+\xi v) \, d\xi$, we have 
$$
\vert g(u+\lambda v)-g(u)-\lambda g^\prime(u)v\vert
\leq \int_0^\lambda |v| |g^\prime(u+\xi v)-g^\prime(u)| \, d\xi
\leq 
[g^\prime]_s \, |v|^{1+s} \, \lambda^{1+s} . 
$$
It follows that 
\begin{equation*}
\vert \ell(y,-\lambda Y^1) \vert=|g(y-\lambda Y^1)-g(y)+\lambda g^{\prime}(y) \, Y^1| \leq  [g^\prime]_s \, \lambda^{1+s} 
|Y^1|^{1+s}
\end{equation*}
and thus, using \eqref{estimateF1Y1},
\begin{multline}\label{estim1}
\big \Vert \ell(y,-\lambda Y^1)\big\Vert_2 \leq  [g^\prime]_s \, \lambda^{1+s} 
\, \big \Vert |Y^1|^{1+s}\big\Vert_{L^2(0,T;L^{2}(\Omega))}
\leq [g^\prime]_s \, \lambda^{1+s} \sqrt{2}C \Vert Y^1\Vert_{L^{\infty}(Q_T)}^{1+s} \\
\leq [g^\prime]_s \, \lambda^{1+s} \sqrt{2}C \Big(Ce^{C\sqrt{\Vert g^{\prime}(y)\Vert_{\infty}}}\Big)^{1+s} E(y,f)^{\frac{1+s}{2}}
\end{multline}
for some positive constant $C$ only depending on $\Omega$ and $T$.
Hence, using \eqref{Eell}, we get
\begin{equation*}
\begin{split}
\sqrt{2E\big((y,f)-\lambda (Y^1,F^1)\big)} 
& \leq \big\Vert (1-\lambda)\big(\partial_{tt}y-\partial_{xx}y+g(y)-f1_\omega\big)\big\Vert_2 +\big\Vert \ell(y,-\lambda Y^1)\big\Vert_2 \\
& \leq  \vert 1-\lambda\vert \sqrt{2E(y,f)}+ [g^\prime]_s \, \lambda^{1+s}
\, \big \Vert |Y^1|^{1+s}\big\Vert_{L^2(0,T;L^{2}(\Omega))} 
\end{split}
\end{equation*}
and, using \eqref{estim1}, the estimate \eqref{estimateE} follows.
\end{proof}

\paragraph{Consequence.}
An important consequence of Proposition \ref{propo1} and in particular of \eqref{ineq_E_Eprime} is that any \textit{critical} point $(y,f)\in \mathcal{A}$ of $E$ (i.e., $E^\prime(y,f)=0$) is a zero of $E$, and thus is a pair solution of the controllability problem. 
Moreover:
\begin{quote}
\emph{given any sequence $(y_k,f_k)_{k\in\N}$ in $\mathcal{A}$ such that $\Vert E^\prime(y_k,f_k)\Vert_{\mathcal{A}_0^\prime}\underset{k\to+\infty}{\longrightarrow} 0$ and such that $\Vert g^{\prime}(y_k)\Vert_{\infty}$ is uniformly bounded, we have $E(y_k,f_k)\underset{k\to+\infty}{\longrightarrow} 0$.} 
\end{quote}
This is thanks to this instrumental property that a minimizing sequence for $E$ cannot be stuck in a local minimum, and this, even though $E$ fails to be convex (it has multiple zeros). Our least-squares algorithm, designed in the next section, and our main result, Theorem \ref{mainthm}, are based on that property.

Note that the left inequality in \eqref{ineq_E_Eprime} indicates the functional $E$ is flat around its zero set. As a consequence, gradient-based minimizing sequences may have a low speed of convergence (see \cite{lemoine-Munch-Pedregal-AMO-20,AM-PP-2014} for such issues for the Navier-Stokes equation).

\subsection{Least-squares algorithm}\label{sec_algo}
Assume that $T>2\max(\ell_1,1-\ell_2)$.
By \eqref{estimateEEprime} in Proposition \ref{propo1}, the vector $-(Y^1,F^1)$, solution of minimal control norm of \eqref{wave-Y1}, is a descent direction for $E$. This leads us to define, for any fixed $m\geq 1$, the sequence $(y_k,f_k)_{k\in\N}$ in $\mathcal{A}$ defined by
\begin{equation}\label{lsq1}
\boxed{
\left\{
\begin{aligned}
&(y_0,f_0) \in \mathcal{A} \\
&(y_{k+1},f_{k+1})=(y_k,f_k)-\lambda_k (Y^1_k,F_k^1) \qquad \forall k\in\N \\
& \lambda_k= \underset{\lambda\in [0,m]}{\mathrm{argmin}} \ E\big((y_k,f_k)-\lambda (Y^1_k,F_k^1)\big)
\end{aligned}
\right.
}
\end{equation}
where $(Y^1_k,F_k^1)\in \mathcal{A}_0$ is the solution of minimal control norm of
\begin{equation}\label{lsq2}
\boxed{
\left\{
\begin{aligned}
& \partial_{tt}Y_k^1 -\partial_{xx} Y^1_k +  g^{\prime}(y_k)\cdot Y^1_k = F^1_k 1_{\omega}+ (\partial_{tt}y_k-\partial_{xx}y_k+g(y_k)-f_k 1_\omega) & \textrm{in}\  Q_T,\\
& Y_k^1=0  & \textrm{on}\  \Sigma_T, \\
& (Y_k^1(\cdot,0),\partial_t Y_k^1(\cdot,0))=(0,0) & \textrm{in}\  \Omega.
\end{aligned}
\right.
}
\end{equation}
The real number $m\geq 1$ is arbitrarily fixed. It is used in the proof of convergence to bound the sequence of optimal descent steps $\lambda_k$. 

\subsection{Main result}
Given any $s\in[0,1]$, we set
\begin{equation}\label{def_betahat}
\beta^0(s) := \frac{s^2}{C^2(2s+1)^2} 
\end{equation}
where $C>0$, only depending on $\Omega$ and $T$, is given by Proposition \ref{propo1}. Note that $(2+\frac{1}{s})C\sqrt{\beta^0(s)}=1$.

\begin{theorem}\label{mainthm}
We assume that $T>2\max(\ell_1,1-\ell_2)$, that $g\in \mathcal{C}^{1,s}(\R)$ for some $s\in[0,1]$, and that there exist $\alpha\geq 0$ and $\beta\in[0,\beta^0(s))$ (with the agreement that $\beta=0$ if $s=0$), such that 
\begin{equation}\label{growth_condition}
\vert g^{\prime}(x)\vert \leq \alpha + \beta \ln^2(1+\vert x\vert)\qquad\forall x\in\R .
\end{equation}
In the case where $s=0$ (i.e., $g^\prime\in L^\infty(\R)$) but $g^\prime\notin \mathcal{C}^{1,s}(\R)$ for any $s\in(0,1]$, we assume moreover that $2 \Vert g^\prime\Vert_\infty C^2 e^{C\sqrt{\Vert g^{\prime}\Vert_\infty}} <1$.
Then:
\begin{itemize}
\item The sequence $(y_k,f_k)_{k\in\N}$ in $\mathcal{A}$ defined by \eqref{lsq1}, initialized at any $(y_0,f_0)\in\A$, converges to $(\overline{y},\overline{f})\in\mathcal{A}$, where $(\overline{y},\overline{f})$ is a solution of \eqref{eq:wave-NL} such that $(y(\cdot,T),\partial_t y(\cdot,T))=(z_0,z_1)$. 
\item The sequence $(\lambda_k)_{k\in\N}$ consists of positive real numbers and converges to $1$.
\item The decreasing sequence $(E(y_k,f_k))_{k\in\N}$ converges to $0$. 
\end{itemize}
Moreover, the convergence of all these sequences is at least linear, and is at least of order $1+s$ after a finite number of iterations.\footnote{We recall that a sequence $(u_k)_{k\in\N}$ of real numbers converges to $0$ with order $\alpha\geq 1$ if there exists $M>0$ such that $\vert u_{k+1}\vert \leq M \vert u_k\vert^\alpha$ for every $k\in\N$. A sequence $(v_k)_{k\in\N}$ of real numbers converges to $0$ at least with order $\alpha\geq 1$ if there exists a sequence $(u_k)_{k\in\N}$ of nonnegative real numbers converging to $0$ with order $\alpha\geq 1$ such that $\vert v_k\vert\leq u_k$ for every $k\in\N$.}
\end{theorem}

\begin{remark}
The limit $(\overline{y},\overline{f})\in\A$ of the sequence $(y_k,f_k)_{k\in\N}$, given by 
$$
(\overline{y},\overline{f})=(y_0,f_0)-\sum_{k=0}^{+\infty}\lambda_k (Y^1_k,F_k^1)  ,
$$
depends on the choice of the initialization $(y_0,f_0)\in\A$ (see also Remark \ref{rem_limit} further). It also depends on the selection criterion that we have chosen: in \eqref{lsq2}, $F^1_k$ is the control of minimal norm.
\end{remark}

\begin{remark}\label{rem2}
In this remark, we assume that $g^\prime\in L^\infty(\R)$.
When $g^\prime$ is not uniformly H\"older continuous, a smallness condition on $\Vert g^\prime\Vert_\infty$ is required in order to obtain the convergence. This condition is not required anymore as soon as $g^\prime\in\mathcal{C}^{1,s}(\R)$ for some $s\in(0,1]$: indeed, then, $g^\prime$ satisfies the growth condition \eqref{growth_condition} with $\alpha=\Vert g^\prime\Vert_\infty$ and $\beta=0$, and Theorem \ref{mainthm} can be applied. 
\end{remark}

\begin{remark}\label{rem_betastar}
In Theorem \ref{mainthm}, we have assumed that the nonnegative coefficient $\beta$ appearing in the growth condition \eqref{growth_condition} is lower than $\beta^0(s)$, i.e., 
$$
\limsup_{\vert x\vert\to+\infty} \frac{\vert g^\prime(x)\vert}{\ln^2\vert x\vert} < \frac{s^2}{(2s+1)^2C^2} 
$$
(with the agreement that $\limsup_{\vert x\vert\to+\infty} \frac{\vert g^\prime(x)\vert}{\ln^2\vert x\vert}=0$ if $s=0$), 
which, of course, implies that
$$
\limsup_{\vert x\vert \to +\infty} \frac{\vert g(x)\vert }{\vert x\vert \ln^2\vert x\vert} < \frac{s^2}{(2s+1)^2C^2} .
$$
The threshold $\beta^0(s)$ is maximal when $s=1$, i.e., when $g^\prime$ is Lipschitz continuous, and we have $\beta^0(1)=\frac{1}{9C^2}$.

In comparison, the threshold $\bar\beta$ in Theorem \ref{thm_Zuazua} satisfies $\overline{\beta}< 1/(1+C)^{2}$ where $C$ is another constant (only depending on $\Omega$ and $T$), appearing in the a priori estimate \eqref{estimate_u_z} of Lemma \ref{controllability_result} in Appendix \ref{sec:linearizedwave}.

There exist cases covered by Theorem \ref{thm_Zuazua} (or, by the extension established in \cite{Cannarsa_Loreti_Komornik2002}), in which exact controllability of \eqref{eq:wave-NL} is true, but that are not covered by Theorem \ref{mainthm}.
Note however that the example 
$g(x)=a+bx+\frac{1}{9C^2+\varepsilon} x \ln^2(1+\vert x\vert)$ for any $\varepsilon > 0$ and any $a,b\in\R$
(which is somehow the limit case in Theorem \ref{thm_Zuazua}) satisfies $g\in \mathcal{C}^{1,1}(\R)$ as well as \eqref{growth_condition}.

While Theorem \ref{thm_Zuazua} was established in \cite{zuazua93} by a nonconstructive Leray-Schauder fixed point argument, we obtain here, in turn, a new proof of the exact controllability of semilinear 1D wave equations, which is moreover constructive, with an algorithm that converges unconditionally, at least with order $1+s$.
\end{remark}

\begin{remark}
The convergence in Theorem \ref{mainthm} is unconditional. Anyway, a natural example of an initialization $(y_0,f_0)\in\A$ is to take $(y_0,f_0)=(y^\star,f^\star)$, the unique solution of minimal control norm of \eqref{eq:wave-NL} with $g=0$ (i.e., in the linear case).
\end{remark}

\begin{remark}\label{rem_number_iter}
As stated in Theorem \ref{mainthm}, the convergence is at least of order $1+s$ after a number $k_0$ of iterations. 
In this remark, we give the precise expression for $k_0$ in function of the various parameters.
Given any $s\in[0,1]$, any $\alpha\geq 0$, any $\beta\in[0,\beta^0(s))$ and any $M\geq 0$, we set
\begin{equation}\label{defc}
c 
:= [g^\prime]_s \, C^{2+s} e^{(1+s)C\sqrt{\alpha}} (1+M)^{(1+s)C\sqrt{\beta}} 
\end{equation}
with the agreement that, when $s=0$, we take $\alpha=\Vert g^{\prime}\Vert_\infty$ and $\beta=0$, so that $c=2 \Vert g^\prime\Vert_\infty C^2  e^{C\sqrt{\Vert g^{\prime}\Vert_\infty}}$ (because, by convention, $[g']_0 :=2\Vert g^\prime\Vert_\infty$), which is the quantity required to be less than $1$ in Theorem \ref{mainthm} when $s=0$. With this convention, $c$ is a continuous function of $s$ on $[0,1]$.

If $(1+s)\, c\, E(y_0,f_0)^{\frac{s}{2}} < 1$ (this includes the case $s=0$) then $k_0=0$, and otherwise,
\begin{equation}\label{defk0}
k_0 = \left\lfloor  \frac{(1+s)^{1+\frac{1}{s}}}{s} \left( c^{\frac{1}{s}} \sqrt{E_0} - 1 \right) \right\rfloor + 1 ,
\end{equation}
where $\lfloor \cdot\rfloor$ is the integer part, and where $M>0$ is the minimal possible real positive number such that
\begin{equation*}
\begin{split}
& 1 \leq (1+s) [g^\prime]_s \, C^{2+s} e^{(1+s)C\sqrt{\alpha}} (1+M)^{(1+s)C\sqrt{\beta}} E_0^{\frac{s}{2}}  , \\
& C \Vert (y_0,f_0)\Vert_\H + \frac{m}{s}(1+s)^{1+\frac{1}{s}} [g^\prime]_s^{\frac{1}{s}} C^{3+\frac{2}{s}} \, e^{(2+\frac{1}{s})C\sqrt{\alpha}} (1+M)^{(2+\frac{1}{s})C\sqrt{\beta}} E_0 \leq M .
\end{split}
\end{equation*}
The real number $M$ is defined in an implicit way. More details are given in Section \ref{sec_proof_mainthm}, at Step \ref{step7} of the proof, where we give in particular an explicit expression for $M$ when $s\simeq 0$.
\end{remark}

\begin{remark}
As a continuation to Remark \ref{rem2}, it is interesting to note that, assuming that $g^\prime\in L^\infty(\R)$:
\begin{itemize}
\item As stated in Theorem \ref{mainthm}, if $g^\prime\notin \mathcal{C}^{1,s}(\R)$ for any $s\in(0,1]$, to obtain convergence it is required to assume that $2 \Vert g^\prime\Vert_\infty C^2 e^{C\sqrt{\Vert g^{\prime}\Vert_\infty}} <1$, i.e., that $c_{\vert s=0} < 1$ with the notations of Remark \ref{rem_number_iter}, and we have $k_0=0$.
\item If $2 \Vert g^\prime\Vert_\infty C^2 e^{C\sqrt{\Vert g^{\prime}\Vert_\infty}} \geq 1$ and if $g^\prime\in \mathcal{C}^{1,s}(\R)$ for some $s\in(0,1]$, then Theorem \ref{mainthm} applies and $k_0$ is given by \eqref{defk0}. Moreover, $k_0$ is larger as $s>0$ is smaller: more precisely, we have $k_0\sim \frac{e}{s} \left( 2 \Vert g^\prime\Vert_\infty C^2 e^{C\sqrt{\Vert g^{\prime}\Vert_\infty}} \right) ^{\frac{1}{s}}$ as $s\rightarrow 0$.
\end{itemize}
\end{remark}

\begin{remark}\label{rem_estimA}
Using \eqref{growth_condition}, we have, for every $(y_0,f_0)\in\A$,
\begin{equation}\label{estimA}
\begin{split}
\sqrt{E(y_0,f_0)}&\leq \Vert \partial_{tt}y_0-\partial_{xx}y_0\Vert_2 +\Vert f_0 1_{\omega}\Vert_2+\Vert g(y_0)\Vert_2\\
&\leq \Vert (y_0,f_0)\Vert_\H+ T \vert g(0)\vert + T (\alpha +\beta \ln^2(1+\Vert y_0\Vert_\infty) )\Vert y_0\Vert_\infty\\
&\leq \Vert (y_0,f_0)\Vert_\H+ T \vert g(0)\vert + T (\alpha +\beta \ln^2(1+\Vert (y_0,f_0)\Vert_\H) )\Vert (y_0,f_0)\Vert_\H .
\end{split}
\end{equation}
\end{remark}

\begin{remark}\label{remproof}
If $s=0$ or if $\beta=0$ in \eqref{growth_condition} then $g^\prime\in L^\infty(\R)$. In this case the proof of Theorem \ref{mainthm} is simpler. When $s>0$ and $\beta>0$, as alluded at the end of Section \ref{sec_lsq_method}, in the proof of Theorem \ref{mainthm}, the main difficulty is to prove that the sequence $(\Vert y_k\Vert_{\infty})_{k\in\N}$ (defined in \eqref{lsq1}) remains uniformly bounded, in particular in order to keep a uniform bound on the sequence of observability constants $Ce^{C\sqrt{\Vert g^{\prime}(y_k)\Vert_{\infty}}}$ appearing in the estimates of Proposition \ref{propo1}. 
In the proof, done in Section \ref{sec_proof_mainthm}, this difficulty is handled by an \emph{a priori} assumption, which we prove to be satisfied \emph{a posteriori} thanks to fine estimates.
\end{remark}

\section{Proof of Theorem \ref{mainthm}}\label{sec_proof_mainthm}
This section is devoted to proving Theorem \ref{mainthm}. We assume that $g\in\mathcal{C}^{1,s}(\R)$ for some $s\in[0,1]$.

\paragraph{Preliminary remark.}
Let $(y_0,f_0)\in \A$ be arbitrarily fixed. In the sequel, we denote by
$$
E_k := E(y_k,f_k) 
\qquad \forall k\in\N.
$$
By the minimization property in the definition \eqref{lsq1} of the algorithm, we have $E_{k+1} = E\big((y_k,f_k)-\lambda_k (Y^1_k,F_k^1)\big) \leq E\big((y_k,f_k)-\lambda(Y^1_k,F_k^1)\big)$ for every $\lambda\in[0,m]$. 
Applying the estimate \eqref{estimateE} of Proposition \ref{propo1}, Item \ref{propo1_estimateE}, to $(y_k,f_k)$, we infer that
\begin{equation}\label{estimprelim}
E_{k+1} \leq \min_{\lambda\in[0,m]} \Big( \vert 1-\lambda\vert + \lambda^{1+s} K(y_k) E_k^\frac{s}{2} \Big)^2 E_k \qquad \forall k\in\N
\end{equation}
where we recall that
$$
K(y_k) 
= [g^\prime]_s \, C^{2+s} \, e^{(1+s)C\sqrt{\Vert g^{\prime}(y_k)\Vert_\infty}} .
$$
The estimate \eqref{estimprelim} is instrumental in the proof of Theorem \ref{mainthm}.

\medskip

Having in mind Remark \ref{remproof}, we fix a constant $M>0$, large enough, to be chosen later. In what follows, we make the \emph{a priori} assumption 
\begin{equation}\label{apriori}
\Vert y_k\Vert_\infty\leq M\qquad\forall k\in\N.
\end{equation}
We are going to see \emph{a posteriori} that, if $M$ is adequately chosen large enough, then \eqref{apriori} is indeed satisfied.
The proof goes in several steps.

\begin{step}\label{step1}
There exists $k_0\in\N$ (given by \eqref{defk0}) such that the sequence $(E_k)_{k\geq k_0}$ decays to $0$ with order greater than or equal to $1+s$. 
\end{step}

Using the growth condition \eqref{growth_condition} and using the \emph{a priori} assumption \eqref{apriori}, we have
$$
\Vert g^\prime(y_k)\Vert_\infty \leq \alpha+\beta\ln^2(1+M) \qquad\forall k\in\N.
$$
Here and in the sequel, we adopt the convention that, when $s=0$, we take $\alpha=\Vert g^\prime\Vert_\infty$ and $\beta=0$.
Using the inequality $\sqrt{a+b}\leq\sqrt{a}+\sqrt{b}$ for all $a,b\geq 0$, we get
\begin{equation}\label{majC}
C e^{C \sqrt{\Vert g^\prime(y_k)\Vert_\infty}} \leq C e^{C\sqrt{\alpha+\beta\ln^2(1+M)}} \leq C e^{C\sqrt{\alpha}} (1+M)^{C\sqrt{\beta}}
\end{equation}
and thus $K(y_k) \leq c$ where $K(y_k)$ is defined by \eqref{estimateE_C} and $c$ is defined by \eqref{defc} (including the case $s=0$).
By \eqref{estimprelim}, we have 
\begin{equation}\label{estek}
\sqrt{E_{k+1}} \leq \min_{\lambda\in [0,m]} e_k(\lambda) \sqrt{E_k} 
\qquad\textrm{with}\qquad
e_k(\lambda):=\vert 1-\lambda\vert +\lambda^{1+s} c\, E_k^{\frac{s}{2}} .
\end{equation}
Let $\widetilde\lambda_k\in[0,m]$ be the minimizer of $e_k(\lambda)$ over $[0,m]$ (not to be confused with $\lambda_k$ defined in \eqref{lsq1}). 

\medskip

Let us first treat the case where $s\in(0,1]$.
Assuming that $E_k>0$ (otherwise there is nothing to do), we have
\begin{equation} \label{ptildek}
\begin{array}{lclcl}
\displaystyle\widetilde\lambda_k = 1 & \textrm{and} & \displaystyle e_k(\widetilde\lambda_k) = c\, E_k^{\frac{s}{2}} & \textrm{if} &  (1+s)^{\frac{1}{s}} c^{\frac{1}{s}}\sqrt{E_k}<1 , \\
\displaystyle\widetilde\lambda_k = \frac{1}{(1+s)^{\frac{1}{s}} c^{\frac{1}{s}}\sqrt{E_k}} & \textrm{and} & \displaystyle e_k(\widetilde\lambda_k) = 1- \frac{s}{(1+s)^{1+\frac{1}{s}} c^{\frac{1}{s}}\sqrt{E_k}} & \textrm{if} &  (1+s)^{\frac{1}{s}} c^{\frac{1}{s}}\sqrt{E_k}\geq 1 , 
\end{array}
\end{equation}
and therefore, by \eqref{estek}, 
\begin{equation}\label{decayEunquarts}
c^{\frac{1}{s}}\sqrt{E_{k+1}}\leq \left\{ \begin{array}{lcl}
\big(c^{\frac{1}{s}}\sqrt{E_k}\big)^{1+s} & \textrm{if} & (1+s)^{\frac{1}{s}} c^{\frac{1}{s}}\sqrt{E_k}<1, \\
c^{\frac{1}{s}}\sqrt{E_k }-\frac{s}{(1+s)^{1+\frac{1}{s}}} & \textrm{if} & (1+s)^{\frac{1}{s}} c^{\frac{1}{s}}\sqrt{E_k} \geq 1.
\end{array}\right.
\end{equation}
\begin{itemize}
\item 
As a first case, let us assume that $0<(1+s)^{\frac{1}{s}} c^{\frac{1}{s}}\sqrt{E_0}< 1$. Then $c^{\frac{1}{s}}\sqrt{E_0}< 1$ and, using \eqref{decayEunquarts}, by iteration, $c^{\frac{1}{s}}\sqrt{E_k}<1$ for every $k\in\N$ and the sequence $(c^{\frac{1}{s}}\sqrt{E_k })_{k\in\N}$ is decreasing. Hence, for every $k\in\N$, we have $(1+s)^{\frac{1}{s}} c^{\frac{1}{s}}\sqrt{E_k}<1$, i.e., we remain in this first case, and since 
$E_{k+1}\leq c^2 E_k^{1+s}$, the sequence $(E_k )_{k\in\N}$ is decreasing and converges to $0$ with order greater than or equal to $1+s$.
\item
As a second case, let us assume that $(1+s)^{\frac{1}{s}} c^{\frac{1}{s}}\sqrt{E_0}\geq 1$. It follows from \eqref{estek} that, as long as $(1+s)^{\frac{1}{s}} c^{\frac{1}{s}}\sqrt{E_k}\geq 1$, we have $c^{\frac{1}{s}}\sqrt{E_k} \leq c^{\frac{1}{s}}\sqrt{E_0} - k \frac{s}{(1+s)^{1+\frac{1}{s}}}$. Hence there exists $k_0\in\N$ 
such that $(1+s)^{\frac{1}{s}} c^{\frac{1}{s}}\sqrt{E_k}< 1$ for every $k\geq k_0$. This means that, after a finite number of iterations, we turn back to the first case. 
The minimal number of iterations is given by the formula \eqref{defk0}.
\end{itemize}


Finally, let us treat the case where $s=0$. The function $e_k(\lambda)$ is piecewise linear, and is increasing whenever $c>1$: this is why we need the smallness condition $c= 2 \Vert g^\prime\Vert_\infty C^2  e^{C\sqrt{\Vert g^{\prime}\Vert_\infty}}c<1$. Thanks to this assumption, the minimizer of $e_k(\lambda)$ is $\widetilde\lambda_k=1$ and thus $e_k(\widetilde\lambda_k)=c$. Hence $E_{k+1} \leq c E_k$ (which is also what we obtain by taking the limit $s\rightarrow 0^+$ in the first case of \eqref{decayEunquarts}) and thus $(E_k)_{k\in\N}$ is decreasing and converges to $0$ at least linearly.
In this case we have $k_0=0$.


\begin{remark}\label{remquotient}
For every $k\geq k_0$, we have (assuming that $E_k>0$)
$$
\frac{E_{k+1}}{E_k} \leq c^2 E_k^{s} .
$$
Since $E_k\to 0$ at least with order $1+s$, it follows that $\frac{E_{k+1}}{E_k}\to 0$ as $k\to+\infty$, at least with order $1+s$. 
Note also that $\lambda_k>0$ for every $k\in\N$ because the sequence $(E_k)_{k\in\N}$ is decreasing.
\end{remark}

\begin{step}\label{step2}
The sequence $(\lambda_k)_{k\geq k_0}$ defined in \eqref{lsq1} converges to $1$ as $k\to+\infty$ at least with order $1+s$.
\end{step}

Applying \eqref{Eell} to $(y,f)=(y_k,f_k)$, $(Y^1,F^1)=(Y^1_k,F^1_k)$ and $\lambda=\lambda_k$, we have, since $\lambda_k\leq m$ (and assuming that $E_k>0$),
\begin{equation*}
\begin{split}
(1-\lambda_k)^2
&=\frac{E_{k+1}}{E_k}-(1-\lambda_k)\frac{\big(\partial_{tt}y_k-\partial_{xx}y_k+g(y_k)-f_k 1_\omega,\ell(y_k,-\lambda_k Y_k^1) \big)_2}{E_k} - \frac{\Vert \ell(y_k,-\lambda_k Y_k^1)\Vert^2_{2}}{2E_k}\\
&\leq \frac{E_{k+1}}{E_k}-(1-\lambda_k)\frac{\big(\partial_{tt}y_k-\partial_{xx}y_k+g(y_k)-f_k 1_\omega,\ell(y_k,-\lambda_k Y_k^1) \big)_2}{E_k}\\
&\leq \frac{E_{k+1}}{E_k}+m\sqrt{2}\,\frac{\Vert \ell(y_k,\lambda_k Y_k^1) \Vert_2}{\sqrt{E_k}}.
\end{split}
\end{equation*}
By \eqref{estim1}, we have $\Vert \ell(y_k,\lambda_k Y_k^1) \Vert_2 \leq \lambda_k^{1+s}\sqrt{2}K(y_k) E_k^{\frac{1+s}{2}} \leq m^{1+s} \sqrt{2}\, c\, E_k^{\frac{1+s}{2}}$, and thus 
$$
(1-\lambda_k)^2\leq \frac{E_{k+1}}{E_k}+2 m^{2+s}\, c\, E_k^{\frac{s}{2}} \qquad\forall k\in\N.
$$
Since $E_k\to 0$ at least with order $1+s$ by Step \ref{step1} and $\frac{E_{k+1}}{E_k}\to 0$ at least with order $1+s$ by Remark \ref{remquotient}, it follows that $\lambda_k\to 1$ at least with order $1+s$.

\begin{step}\label{step3}
We have $e_0(\widetilde\lambda_0)<1$, and the sequence $(e_k(\widetilde\lambda_k))_{k\in\N}$ decays to $0$.
\end{step}

Indeed, since $e_0(0)=1$ and $e_0^\prime(0)<0$, we have $e_0(\widetilde\lambda_0)=\min_{\lambda\in [0,m]}e_0(\lambda)<1$ (also in the case where $s=0$ thanks to the smallness condition). The rest of the statement follows from \eqref{ptildek}.

\begin{step}\label{step4}
The series $\sum_{k\geq 0} \sqrt{E_k}$ converges, and $\displaystyle\sum_{k=p}^{+\infty} \sqrt{E_k} \leq \frac{1}{1-e_0(\widetilde\lambda_0)} \sqrt{E_p}$ for every $p\in\N$.
\end{step}

The fact that the series $\sum_{k\geq 0} \sqrt{E_k}$ converges already follows from Remark \ref{remquotient} since $\frac{E_{k+1}}{E_k}\to 0$. We will however use the rough estimate stated here. To prove it, we observe that, since the sequence $(e_k(\widetilde\lambda_k))_{k\in\N}$ is decreasing, we have $e_k(\widetilde\lambda_k)\leq e_p(\widetilde\lambda_p)\leq e_0(\widetilde\lambda_0)<1$ for all $k,p\in\N$ such that $k\geq p$, and we infer from \eqref{estek} that
$$
\sqrt{E_k}\leq (e_0(\widetilde\lambda_0))^{k-p}\sqrt{E_p} \qquad \forall k,p\in\N,\ k\geq p
$$
and the result follows.

\begin{step}\label{step5}
The series $\sum_{k\geq 0}\lambda_k (Y^1_k,F_k^1)$ converges in $\A_0$, and
$$
\bigg\Vert \sum_{k=p}^{q}  \lambda_k (Y^1_k,F^1_k) \bigg\Vert_\H \leq m C  e^{C\sqrt{\alpha}} (1+M)^{C\sqrt{\beta}}  \frac{1}{1-e_0(\widetilde\lambda_0)} \sqrt{E_p}  \qquad \forall p,q\in\N,\ q\geq p .
$$
\end{step}

Since $\lambda_k\leq m$, it follows from \eqref{estimateF1Y1_A0} in Proposition \ref{propo1} and from \eqref{majC} that
$$
\lambda_k \Vert (Y^1_k,F^1_k)\Vert_\H \leq m C  e^{C\sqrt{\alpha}} (1+M)^{C\sqrt{\beta}}  \sqrt{E_k}\qquad \forall k\in\N
$$
and the result follows, using Step \ref{step4}.

\begin{step}\label{step6}
The sequence $(y_k,f_k)_{k\in\N}$ defined by \eqref{lsq1} converges to the element $(\overline{y},\overline{f})\in\A$ given by
$$
(\overline{y},\overline{f})=(y_0,f_0)-\sum_{k=0}^{+\infty}\lambda_k (Y^1_k,F_k^1) 
$$
and the convergence is at least of order $1+s$ after $k_0$ iterations (where $k_0$ is given by Step \ref{step1}).
Moreover, $(\overline{y},\overline{f})$ is a solution of \eqref{eq:wave-NL} such that $(y(\cdot,T),\partial_t y(\cdot,T))=(z_0,z_1)$. 
\end{step}

Indeed, by \eqref{lsq1}, we have $(y_n,f_n)=(y_0,f_0)-\sum_{k=0}^{n-1} \lambda_k (Y^1_k,F_k^1)$, hence $(y_k,f_k)$ converges to $(\overline{y},\overline{f})$ defined above. 
Let us prove that $\bar f$ is a null control for $\bar y$ solution of \eqref{eq:wave-NL}.
Using that $(Y^1_k, F^1_k)\in\mathcal{A}_0$ converges to zero as $k\to+\infty$, passing to the limit in \eqref{lsq2}, we infer that $(\overline{y},\overline{f})\in \mathcal{A}$ solves 
\begin{equation}
\label{wave-limit}
\left\{
\begin{aligned}
& \partial_{tt}\overline{y} - \partial_{xx}\overline{y} +  g(\overline{y}) = \overline{f} 1_{\omega} &\textrm{in}\quad Q_T,\\
& \overline{y}=0 & \textrm{on}\  \Sigma_T, \\
& (\overline{y}(\cdot,0),\partial_t\overline{y}(\cdot,0))=(y_0,y_1) & \textrm{in}\  \Omega.
\end{aligned}
\right.
\end{equation}
Since $(\overline{y},\overline{f})\in\mathcal{A}$, we have $(\overline{y}(\cdot,T),\partial_t\overline{y}(\cdot,T))=(z_0,z_1)$ in $\Omega$, i.e., $\bar f$ is a control for $\bar y$ solution of \eqref{eq:wave-NL}.
Now, for every $k\in\N$, we have
\begin{equation}\label{estim_coercivity}
\Vert (\overline{y},\overline{f})-(y_n,f_n)\Vert_{A_0} 
= \Big\Vert \sum_{k=n}^{+\infty} \lambda_k (Y^1_k,F^1_k)\Big\Vert_\H
\leq  \frac{m C}{1-e_0(\widetilde\lambda_0)}  e^{C\sqrt{\alpha}} (1+M)^{C\sqrt{\beta}} \sqrt{E_n} 
\end{equation}
The convergence to $0$ with order greater than or equal to $1+s$ after a finite number of iterations follows from Lemma \ref{step1}.

\begin{remark}\label{rem_limit}
The estimate \eqref{estim_coercivity} is a kind of coercivity property for the functional $E$.
We emphasize, in view of the non-uniqueness of the zeros of $E$, that an estimate (similar to \eqref{estim_coercivity}) of the form  $\Vert (\overline{y},\overline{f})-(y,f)\Vert_\H \leq C \sqrt{E(y,f)}$ does not hold for every $(y,f)\in\mathcal{A}$. We also insist on the fact that the sequence $(y_k,f_k)_{k\in\N}$ and its limit $(\overline{y},\overline{f})$ are uniquely determined by the initialization $(y_0,f_0)$ and by our selection criterion for the control $F^1$. 
\end{remark}

\begin{step}\label{step7}
If $(2+\frac{1}{s})C\sqrt{\beta}<1$ whenever $s\in(0,1]$, and if $2 \Vert g^\prime\Vert_\infty C^2 e^{C\sqrt{\Vert g^{\prime}\Vert_\infty}} < 1$ whenever $s=0$, then there exists $M>0$ sufficiently large (depending on the initialization $(y_0,f_0)$, on $\alpha$ and on $\beta$) such that the \emph{a priori} assumption \eqref{apriori} is indeed satisfied.
\end{step}

Let us summarize what we have done, under the growth condition \eqref{growth_condition} and under the \emph{a priori} assumption \eqref{apriori}.
By \eqref{lsq1}, we have $(y_k,f_k)=(y_0,f_0)-\sum_{j=0}^{k-1} \lambda_j (Y^1_j,F^1_j)$, and then, using \eqref{estiminftyA} and Step \ref{step5}, we get the \emph{a posteriori} estimate
\begin{equation}\label{estiyk}
\Vert y_k\Vert_\infty \leq C \Vert (y_k,f_k)\Vert_\H
\leq C \Vert (y_0,f_0)\Vert_\H + m C^2 e^{C\sqrt{\alpha}} (1+M)^{C\sqrt{\beta}}  \frac{\sqrt{E_0} }{1-e_0(\widetilde\lambda_0)}  
\end{equation}
with the agreement that $\alpha=\Vert g^{\prime}\Vert_\infty$ and $\beta=0$ if $s=0$.
Hence, to prove that the \emph{a priori} assumption \eqref{apriori} is satisfied, it suffices to choose $M>0$ large enough so that the right-hand side of \eqref{estiyk} is less than or equal to $M$.

Recalling that $c = [g^\prime]_s \, C^{2+s} e^{(1+s)C\sqrt{\alpha}} (1+M)^{(1+s)C\sqrt{\beta}}$ for $s\in(0,1]$ and $c=2 \Vert g^\prime\Vert_\infty C^2  e^{C\sqrt{\Vert g^{\prime}\Vert_\infty}}$ for $s=0$,
we infer from the proof of Step \ref{step1} (in particular, from \eqref{ptildek}) that:

\medskip

\noindent $\bullet$\ \ 
For any $s\in[0,1]$, if 
\begin{equation}\label{provided<1}
(1+s) \, c \, E_0^{\frac{s}{2}} < 1
\end{equation}
then $e_0(\widetilde\lambda_0) = c E_0^{\frac{s}{2}}$ and thus, by \eqref{estiyk},
\begin{equation}\label{aposteriori_s>0_petit}
\Vert y_k\Vert_\infty 
\leq C \Vert (y_0,f_0)\Vert_\H +  C^{\frac{s}{1+s}} \, \frac{m}{[g^\prime]_s^{\frac{1}{1+s}}} c^{\frac{1}{1+s}} \frac{\sqrt{E_0} }{1- c E_0^{\frac{s}{2}}} .
\end{equation}
These estimates include the case $s=0$ (with $[g^\prime]_0 = 2\Vert g^{\prime}\Vert_\infty$), in which \eqref{provided<1} is exactly the smallness condition $2 \Vert g^\prime\Vert_\infty C^2  e^{C\sqrt{\Vert g^{\prime}\Vert_\infty}} < 1$.

Here, we choose the minimal real number $M>0$ such that
\begin{equation}\label{M_provided<1}
C \Vert (y_0,f_0)\Vert_\H +  C^{\frac{s}{1+s}} \, \frac{m}{[g^\prime]_s^{\frac{1}{1+s}}} c^{\frac{1}{1+s}} \frac{\sqrt{E_0} }{1- c E_0^{\frac{s}{2}}}   \leq  M . 
\end{equation}
This is possible by assuming that $E_0$ is sufficiently small, because then, there exist real numbers $M$ (which cannot be arbitrarily large) satisfying both \eqref{provided<1} and \eqref{M_provided<1}. This observation follows by inspecting both inequalities, either with $M$ large, or with $E_0$ small.

The above choice of $M$ is implicit and unfortunately cannot be made explicit for any $s$. We can however give explicit expressions when $s\to 0$, as follows. For $s=0$, \eqref{provided<1} is written as $c = 2 \Vert g^\prime\Vert_\infty C^2  e^{C\sqrt{\Vert g^{\prime}\Vert_\infty}} < 1$ (smallness condition) and \eqref{M_provided<1} gives
$$
M_{\vert s=0} =
C \Vert (y_0,f_0)\Vert_\H +   \frac{m \,C^2\,  e^{C\sqrt{\Vert g^{\prime}\Vert_\infty}}}{1- 2 \Vert g^\prime\Vert_\infty C^2  e^{C\sqrt{\Vert g^{\prime}\Vert_\infty}}} \sqrt{E_0}  . 
$$
Now, when $s\to 0$, $s>0$, we must have $\beta\rightarrow 0$, and then, taking equivalents, \eqref{provided<1} gives
$$
[g^\prime]_s C^2 e^{C\sqrt{\alpha}} E_0^{\frac{s}{2}} < 1
$$
while \eqref{M_provided<1} gives
$$
M_{s\simeq 0} = C \Vert (y_0,f_0)\Vert_\H +  \frac{m \,C^2\,  e^{C\sqrt{\alpha}}}{1- [g^\prime]_s C^2  e^{C\sqrt{\alpha}}} \sqrt{E_0}
$$
which is in accordance with the case $s=0$.

With this choice, \eqref{provided<1} gives a smallness condition on the initialization when $s>0$.

\medskip

\noindent $\bullet$\ \ 
For any $s\in(0,1]$, if
\begin{equation}\label{provided>1}
(1+s) \, c \, E_0^{\frac{s}{2}} \geq 1
\end{equation}
(now $s=0$ is excluded) then
$\frac{1}{1-e_0(\widetilde\lambda_0)} = \frac{(1+s)^{1+\frac{1}{s}}}{s} c^{\frac{1}{s}} \sqrt{E_0}$
and thus, by \eqref{estiyk},
\begin{equation}\label{aposteriori_s>0}
\Vert y_k\Vert_\infty 
\leq C \Vert (y_0,f_0)\Vert_\H + C^{\frac{s}{1+s}} \,  \frac{m(1+s)^{1+\frac{1}{s}}}{s [g^\prime]_s^{\frac{1}{1+s}} } \, c^{\frac{1+2s}{s(1+s)}}  E_0 .
\end{equation}
Here, we choose $M>0$ large enough such that
\begin{equation}\label{M_provided>1}
C \Vert (y_0,f_0)\Vert_\H + C^{\frac{s}{1+s}} \,  \frac{m(1+s)^{1+\frac{1}{s}}}{s [g^\prime]_s^{\frac{1}{1+s}} } \, c^{\frac{1+2s}{s(1+s)}}  E_0 \leq M .
\end{equation}
This is possible because there exist (large) real numbers $M$ satisfying \eqref{M_provided>1}. Indeed, taking $M$ large, \eqref{M_provided>1} is of the kind 
$$
\mathrm{Cst}\, M^{(2+\frac{1}{s})C\sqrt{\beta}} E_0 \lesssim M
$$
which has solutions because, by assumption, $(2+\frac{1}{s})C\sqrt{\beta} < 1$.
More precisely, here, we choose the minimal real number $M>0$ such that
\begin{equation*}
\begin{split}
& 1 \leq (1+s) [g^\prime]_s \, C^{2+s} e^{(1+s)C\sqrt{\alpha}} (1+M)^{(1+s)C\sqrt{\beta}} E_0^{\frac{s}{2}}  , \\
& C \Vert (y_0,f_0)\Vert_\H + \frac{m}{s}(1+s)^{1+\frac{1}{s}} [g^\prime]_s^{\frac{1}{s}} C^{3+\frac{2}{s}} \, e^{(2+\frac{1}{s})C\sqrt{\alpha}} (1+M)^{(2+\frac{1}{s})C\sqrt{\beta}} E_0 \leq M .
\end{split}
\end{equation*}
As before, the above choice of $M$ is implicit. We can anyway give explicit formulas when $s\to 0$. Indeed, when $s\simeq 0$, $s>0$, we have then $\beta\to 0$, $\frac{C\sqrt{\beta}}{s}<1$, and \eqref{provided>1} gives
$$
[g^\prime]_s C^2 e^{C\sqrt{\alpha}} M^{C\sqrt{\beta}} E_0^{\frac{s}{2}} \geq 1
$$
while \eqref{M_provided>1} gives
$$
M^{1-\frac{C\sqrt{\beta}}{s}} \geq 
C \Vert (y_0,f_0)\Vert_\H + \frac{me}{s} [g^\prime]_s^{\frac{1}{s}} C^{\frac{2}{s}} \, e^{\frac{1}{s}C\sqrt{\alpha}}  E_0 
$$
and therefore
$$
M_{\vert s\simeq 0} = \max \left(
\frac{1}{ \left( [g^\prime]_s C^2 e^{C\sqrt{\alpha}} E_0^{\frac{s}{2}} \right)^{1/C\sqrt{\beta}} } \ , \ \ 
\left( C \Vert (y_0,f_0)\Vert_\H + \frac{me}{s} [g^\prime]_s^{\frac{1}{s}} C^{\frac{2}{s}} \, e^{\frac{1}{s}C\sqrt{\alpha}}  E_0  \right)^{s / (s-C\sqrt{\beta})}
\right) .
$$

\section{Conclusion and further comments} \label{sec:remarks}
Exact controllability of \eqref{eq:wave-NL} has been established in \cite{zuazua93}, under a growth condition on $g$, by means of a Leray-Schauder fixed point argument that is not constructive. In this paper, under a slightly stronger growth condition and under the additional assumption that $g^\prime$ is uniformly H\"older continuous with exponent $s\in[0,1]$, we have designed an explicit algorithm and proved its convergence of a controlled solution of \eqref{eq:wave-NL}. Moreover, the convergence is super-linear of order greater than or equal to $1+s$ after a finite number of iterations. 
In turn, our approach gives a new and constructive proof of the exact controllability of \eqref{eq:wave-NL}. 

\medskip

Several comments are in order.

\paragraph{Minimization functional.}
Among all possible admissible controlled pair $(y,v)\in \A_0$, we have selected the solution $(Y_1,F_1)$ of \eqref{wave-Y1} that minimizes the functional $J(v)=\Vert v\Vert^2_{2,q_T}$. This choice has led to the estimate \eqref{estimateF1Y1} which is one of the key points of the convergence analysis. The analysis remains true when one considers the quadratic functional $J(y,v)=\Vert w_1  v\Vert^2_{2,q_T} + \Vert w_2  y\Vert^2_2$ for some positive weight functions $w_1$ and $w_2$ (see for instance \cite{cindea_efc_munch_2013}).

\paragraph{Newton method.}
Defining $F:\mathcal{A}\to L^2(Q_T)$ by $F(y,f):=(\partial_{tt}y-\partial_{xx}y + g(y)-f1_\omega)$, we have $E(y,f)=\frac{1}{2}\Vert F(y,f)\Vert_{L^2(Q_T)}^2$ and we observe that, for $\lambda_k=1$, the algorithm \eqref{lsq1}-\eqref{lsq2} coincides with the Newton algorithm applied to $F$ (see \ref{Newton-nn}). This explains the super-linear convergence property obtained in Theorem \ref{mainthm}, in particular the quadratic convergence when $s=1$. Optimizing the parameter $\lambda_k$ gives a global convergence property of the algorithm and leads to the so-called damped Newton method applied to $F$. For this method, global convergence is usually achieved with linear order under general assumptions (see \cite[Theorem 8.7]{deuflhard}). As far as we know, damped type Newton methods have been little applied to partial differential equations in the literature. We mention \cite{lemoinemunch_time, saramito} in the context of fluid mechanics.   

\paragraph{Another variant.}
To simplify, let us take $\lambda_k=1$, as in the standard Newton method. Then, for each $k\in\N$, the optimal pair $(Y_k^1,F_k^1)\in \mathcal{A}_0$ is such that the element $(y_{k+1},f_{k+1})$ minimizes over $\A$ the functional $(z,v)\to J(z-y_k,v-f_k)$ with $J(z,v):=\Vert v\Vert_{2,q_T}$ (control of minimal $L^2(q_T)$ norm). 
Alternatively, we may select the pair $(Y_k^1,F_k^1)$ so that the element $(y_{k+1},f_{k+1})$ minimizes the functional $(z,v)\to J(z,v)$. This leads to the sequence $(y_k,f_k)_{k\in\N}$ defined by 
 \begin{equation}
\label{eq:wave_lambdakequal1}
\left\{
\begin{aligned}
& \partial_{tt}y_{k+1} - \partial_{xx}y_{k+1} +  g^{\prime}(y_k) y_{k+1} = f_{k+1} 1_{\omega}+g^\prime(y_k)y_k-g(y_k) & \textrm{in}\  Q_T,\\
& y_k=0,  & \textrm{on}\  \Sigma_T, \\
& (y_{k+1}(\cdot,0), \partial_t y_{k+1}(\cdot,0))=(u_0,u_1) & \textrm{in}\  \Omega.
\end{aligned}
\right.
\end{equation}
In this case, for every $k\in\N$, $(y_k,f_k)$ is a controlled pair for a linearized wave equation, while, in the case of the algorithm \eqref{lsq1}-\eqref{lsq2}, $(y_k,f_k)$ is a sum of controlled pairs $(Y^1_j,F^1_j)$ for $0\leq j\leq k$.
This formulation used in \cite{EFC-AM-2012} is different and the convergence analysis (at least in the least-squares setting) does not seem to be straightforward because the term $g^\prime(y_k)y_k-g(y_k)$ is not easily bounded in terms of $\sqrt{E(y_k,f_k)}$.

\paragraph{Local controllability when removing the growth condition \eqref{growth_condition}.}
Let us remove the growth condition \eqref{growth_condition} on $g^\prime$. We have the following convergence result, under the assumption that $E(y_0,f_0)$ is small enough.

\begin{proposition}\label{prop_smallness}
Assume that $g\in \mathcal{C}^{1,s}(\R)$ for some $s\in (0,1]$. 
There exists $C([g^\prime]_s)>0$ such that, if $E(y_0,f_0)\leq C([g^\prime]_s)$, then the sequence $(y_k,f_k)_{k\in \N}$ in $\mathcal{A}$ defined in \eqref{lsq1} converges to $(\overline{y},\overline{f})\in\mathcal{A}$, where $\overline{f}$ is a null control for $\overline{y}$ solution of  \eqref{eq:wave-NL}. Moreover, there exists $k_0\in \N$ such that the sequence $(\Vert (\overline{y},\overline{f})-(y_k,f_k)\Vert_\H )_{k\geq k_0}$ is decreasing and converges to $0$ at least with order $1+s$.
\end{proposition}

The proof is a variant of the arguments given in this paper. We do not provide any details.
In the case $g(0)=0$, the smallness assumption on $E(y_0,f_0)$ is satisfied as soon as $\Vert(u_0,u_1)\Vert_{\boldsymbol{V}}$  is small. Therefore, the convergence result stated in Proposition \ref{prop_smallness} is equivalent to the local controllability property for \eqref{eq:wave-NL}.
Proposition \ref{prop_smallness} can actually be seen as a consequence of the usual convergence of the Newton method: when $E(y_0,f_0)$ is small enough, i.e., when the initialization is close enough to the solution, then $\lambda_k=1$ for every $k\in\N$ and we recover the standard Newton method.

\paragraph{Multi-dimensional case.}
Let $\Omega$ is a bounded subset of $\R^d$, $1\leq d\leq 3$, and let $\omega$ be a nonempty open subset of $\Omega$. Assume that the triple $(\Omega,\omega,T)$ satisfies the multiplier condition introduced in \cite{JLL88}. 
Then, we conjecture that Theorem \ref{mainthm} remains true in this context, strengthening the growth condition on $g$ into $\vert g^{\prime}(x)\vert \leq \alpha + \beta \ln^{1/2}(1+\vert x\vert)$, for every $x\in \R$ and some $\beta>0$ small enough. Establishing the result should require to use estimates of \cite{Li_Zhang_2000, Zhang_2000} (see also \cite{lemoine_gayte_munch} for the case of a semilinear heat equations).

\paragraph{Boundary control.}
In this paper we have taken internal controls. Our approach may also be extended, with few modifications, to boundary controls considered in particular in \cite{zuazua91}. We leave this issue open.

\appendix

\section{Appendix: controllability results for the linearized wave equation}\label{sec:linearizedwave}

We recall in this appendix some a priori estimates for the linearized wave equation with potential in $L^\infty(Q_T)$ and source term in $L^2(Q_T)$. 

\begin{lemma}\label{controllability_result}
Let $A\in L^{\infty}(Q_T)$, let $B\in L^2(Q_T)$ and let $(z_0,z_1)\in \boldsymbol{V}$. Assume that $T>2\max(\ell_1,1-\ell_2)$. There exists $u\in L^2(q_T)$ such that the solution of 
\begin{equation}
\label{wave_z}
\left\{
\begin{aligned}
& \partial_{tt}z - \partial_{xx}z +  A z  = u 1_{\omega} + B & \textrm{in}\  Q_T,\\
& z=0 & \textrm{on}\  \Sigma_T, \\
& (z(\cdot,0),\partial_t z(\cdot,0))=(z_0,z_1) &\textrm{in}\  \Omega,
\end{aligned}
\right.
\end{equation}
satisfies $(z(\cdot,T),\partial_t z(\cdot,T))=(0,0)$ in $\Omega$. Moreover, the unique control $u$ minimizing the $L^2(q_T)$ norm and its corresponding solution $z$ satisfy 
\begin{equation}\label{estimate_u_z}
\Vert u\Vert_{2,q_T} + \Vert (z,\partial_t z)\Vert_{L^\infty(0,T;\boldsymbol{V})}\leq C \Big(\Vert B\Vert_2 \, e^{(1+C)\sqrt{\Vert A\Vert_\infty}} + \Vert z_0,z_1\Vert_{\boldsymbol{V}}\Big) e^{C\sqrt{\Vert A\Vert_\infty}}
\end{equation}
for some constant $C>0$ only depending on $\Omega$ and $T$.
\end{lemma}

\begin{proof}
The proof is based on estimates obtained in \cite{zuazua93}. The control of minimal $L^2(q_T)$ norm is given by $u=\ph 1_{\omega}$
where $\ph$ solves the adjoint equation 
\begin{equation}
\label{wave_adjoint}
\left\{
\begin{aligned}
& \partial_{tt}\ph - \partial_{xx}\ph +  A \ph  = 0 & \textrm{in}\  Q_T,\\
& \ph=0 & \textrm{on}\  \Sigma_T, \\
& (\ph(\cdot,0),\partial_t\ph(\cdot,0))=(\ph_0,\ph_1) &\textrm{in}\  \Omega,
\end{aligned}
\right.
\end{equation}
where $(\ph_0,\ph_1)\in \boldsymbol{H}:=L^2(\Omega)\times H^{-1}(\Omega)$ is the unique minimizer of 
$$
J(\ph_0,\ph_1):=\frac{1}{2}\int\!\!\!\!\int_{q_T} \ph^2 +\int\!\!\!\!\int_{Q_T} B \ph - \langle(z_0,z_1),(\ph_0,\ph_1)\rangle_{\boldsymbol{V},\boldsymbol{H}}
$$
with $\langle(z_0,z_1),(\ph_0,\ph_1)\rangle_{\boldsymbol{V},\boldsymbol{H}}:=\langle z_0,\ph_1\rangle_{H_0^1(\Omega),H^{-1}(\Omega)}-(z_1,\ph_0)_{L^2(\Omega),L^2(\Omega)}$. In particular, the control $v$ satisfies the optimality condition
$$
\int\!\!\!\!\int_{q_T} \ph\,\overline{\ph} +\int\!\!\!\!\int_{Q_T} B \overline{\ph} - \langle(z_0,z_1),(\overline{\ph}_0,\overline{\ph}_1)\rangle_{\boldsymbol{V},\boldsymbol{H}}=0 \qquad\forall (\overline{\ph}_0,\overline{\ph}_1)\in \boldsymbol{H}
$$
from which we deduce that $\Vert u\Vert_{2,q_T}^2 \leq \Vert B\Vert_2 \Vert \ph \Vert_2 + \Vert (z_0,z_1)\Vert_{\boldsymbol{V}}\Vert (\ph_0,\ph_1)\Vert_{\boldsymbol{H}}$.
From \cite[Lemma 2]{zuazua93}, we get 
$$
\Vert (\ph,\partial_t\ph)\Vert^2_{L^\infty(0,T;\boldsymbol{H})}\leq B_1 \Vert (\ph_0,\ph_1)\Vert^2_{\boldsymbol{H}} (1+\Vert A\Vert^2_\infty)e^{B_2 \sqrt{\Vert A\Vert_{\infty}}}
$$
for some constants $B_1,B_2>0$, and it follows that 
$\Vert \ph \Vert^2_2\leq T B_1 \Vert (\ph_0,\ph_1)\Vert^2_{\boldsymbol{H}} (1+\Vert A\Vert^2_{\infty})e^{B_2 \sqrt{\Vert A\Vert_{\infty}}}$.
Moreover, from \cite[Theorem 4]{zuazua93}, there exists $C>0$ such that $\Vert (\ph_0,\ph_1)\Vert^2_{\boldsymbol{H}}\leq Ce^{C \sqrt{\Vert A\Vert_{\infty}}}\Vert \ph\Vert^2_{2,q_T}$.
Combining these inequalities, we get 
$$
\Vert u\Vert_{L^2(q_T)}\leq \Big(\Vert B\Vert_2 \sqrt{T}\sqrt{B_1}(1+\Vert A\Vert^2_\infty)^{1/2} e^{\frac{B_2}{2}\sqrt{\Vert A\Vert_\infty}}+\Vert (z_0,z_1)\Vert_{\boldsymbol{V}}\Big)\sqrt{C}e^{\frac{C}{2}\sqrt{\Vert A\Vert_{\infty}}}.
$$
Using the inequality $(1+s^2)^{1/2} \leq e^{\sqrt{s}}$ for every $s\geq 0$, we get the result. 
Then, from \cite[Lemma 1]{zuazua93}, we have
$$
\Vert (z,\partial_t z)\Vert^2_{L^\infty(0,T;\boldsymbol{V})}\leq D_1\Big(  \Vert (z_0,z_1)\Vert^2_{\boldsymbol{H}}(1+\Vert A\Vert_\infty)+ \Vert u 1_{\omega}+B\Vert^2_2\Big)
e^{D_2\sqrt{\Vert A\Vert_\infty}}
$$
for some constants $D_1,D_2>0$, and we infer that 
$$
\Vert (z,\partial_t z)\Vert^2_{L^\infty(0,T;\boldsymbol{V})}\leq D_1\Big(  \Vert (z_0,z_1)\Vert^2_{\boldsymbol{H}}(2+\Vert A\Vert_\infty)+ 2\Vert B\Vert^2_2\Big(1+TB_1\big(1+\Vert A\Vert_{\infty}^2\big)e^{B_2\sqrt{\Vert A\Vert_{\infty}}}\Big) \Big)
e^{D_2\sqrt{\Vert A\Vert_\infty}}.
$$
Using that $(1+s)^{1/2} \leq e^{\sqrt{s}}$ and  $(1+s^2) \leq e^{2\sqrt{s}}$ for every $s\geq 0$, we get the estimate.  
\end{proof}

We next discuss some properties of the operator $K: L^\infty(Q_T)\to L^\infty(Q_T)$ defined by $K(\xi)=y_{\xi}$, a null controlled solution of the linear boundary value problem \eqref{NL_z} with the control $f_{\xi}$ of minimal $L^2(q_T)$ norm. Lemma \ref{controllability_result} with $B=-g(0)$ gives
\begin{equation}
\label{estimateK}
\Vert (y_{\xi},\partial_t y_{\xi})\Vert_{L^\infty(0,T;\boldsymbol{V})}\leq C\Big(\Vert u_0,u_1\Vert_{\boldsymbol{V}}+\Vert g(0)\Vert_2 e^{(1+C)\sqrt{\Vert \widehat{g}(\xi)\Vert_{\infty}}} \Big)e^{C\sqrt{\Vert \widehat{g}(\xi)\Vert_\infty}}.
\end{equation}
As in \cite{zuazua93}, the growth condition \eqref{growth_condition} implies that there exists $d>0$ such that $\Vert \widehat{g}(y)\Vert_\infty\leq d+\beta \ln^2(1+\Vert y\Vert_\infty)$
for every $y\in L^\infty(Q_T)$, and it follows that $e^{C\sqrt{\Vert \hat{g}(\xi)\Vert_\infty}}\leq e^{C\sqrt{d}} (1+\Vert \xi\Vert_{\infty})^{C\sqrt{\beta}}$. Using \eqref{estimateK}, we infer that
$$
 \Vert y_{\xi}\Vert_{\infty}\leq C\Big(\Vert u_0,u_1\Vert_{\boldsymbol{V}}+\Vert g(0)\Vert_2\Big)e^{(1+2C)\sqrt{d}} (1+\Vert \xi\Vert_{\infty})^{(1+2C)\sqrt{\beta}}.
$$
Taking $\beta$ small enough so that $(1+2C)\sqrt{\beta}<1$, we conclude that there exists $M>0$ such that $\Vert \xi\Vert_\infty\leq M$ implies $\Vert K(\xi)\Vert_\infty\leq M$. This is the argument of \cite{zuazua93}. Note that, in contrast to $\beta$, $M$ depends on $\Vert u_0,u_1\Vert_{\boldsymbol{V}}$ (and increases with $\Vert u_0,u_1\Vert_{\boldsymbol{V}}$).

The following result gives an estimate of the difference of two controlled solutions. 

\begin{lemma}\label{gap}
Let $a,A\in L^\infty(Q_T)$ and let $B\in L^2(Q_T)$. Let $u$ and $v$ be the null controls of minimal $L^2(q_T)$ norm for $y$ and $z$ respectively solutions of 
\begin{equation}
\label{wave_y}
\left\{
\begin{aligned}
& \partial_{tt}y - \partial_{xx}y +  A y  = u 1_{\omega} + B & \textrm{in}\  Q_T,\\
& y=0 & \textrm{on}\  \Sigma_T, \\
& (y(\cdot,0),\partial_t y(\cdot,0))=(u_0,u_1) &\textrm{in}\  \Omega,
\end{aligned}
\right.
\end{equation}
and 
\begin{equation}
\label{wave_zz}
\left\{
\begin{aligned}
& \partial_{tt}z - \partial_{xx}z +  (A+a)z   = v 1_{\omega} + B & \textrm{in}\  Q_T,\\
& z=0 & \textrm{on}\  \Sigma_T, \\
& (z(\cdot,0),\partial_t z(\cdot,0))=(u_0,u_1) &\textrm{in}\  \Omega .
\end{aligned}
\right.
\end{equation}
Then, 
$$
\begin{aligned}
\Vert y-z\Vert_{L^\infty(Q_T)} 
& \leq C^4 \Vert a \Vert_\infty e^{C\sqrt{\Vert A+a\Vert_\infty}} e^{(2+3C)\sqrt{\Vert A\Vert_\infty}}\Big(\Vert B\Vert_2 e^{(1+C)\sqrt{\Vert A\Vert_\infty}} + \Vert u_0,u_1\Vert_{\boldsymbol{V}}\Big)
\end{aligned}
$$
for some constant $C>0$ only depending on $\Omega$ and $T$.
\end{lemma}

\begin{proof} 
The controls of minimal $L^2$ norm for $y$ and $z$ are given by $u=\ph 1_\omega$ and 
$v=\ph_a 1_\omega$, where $\ph$ and $\ph_a$ respectively solve the adjoint equations 
\begin{equation*}
\left\{
\begin{aligned}
& \partial_{tt}\ph - \partial_{xx}\ph +  A \ph   = 0 & \textrm{in}\  Q_T,\\
& \ph=0 & \textrm{on}\  \Sigma_T, \\
& (\ph(\cdot,0),\partial_t\ph(\cdot,0))=(\ph_0,\ph_1) &\textrm{in}\  \Omega,
\end{aligned}
\right.
\qquad\quad 
\left\{
\begin{aligned}
& \partial_{tt}\ph_{a} - \partial_{xx}\ph_{a} +  (A+a) \ph_a   = 0 & \textrm{in}\  Q_T,\\
& \ph=0 & \textrm{on}\  \Sigma_T, \\
& (\ph(\cdot,0),\partial_t\ph(\cdot,0))=(\ph_{a,0},\ph_{a,1}) &\textrm{in}\  \Omega,
\end{aligned}
\right.
\end{equation*}
for some appropriate $(\ph_0,\ph_1), (\ph_{a,0},\ph_{a,1})\in \boldsymbol{H}$. Hence $Z:=z-y$ solves 
\begin{equation}
\label{diffZ}
\left\{
\begin{aligned}
& \partial_{tt}Z - \partial_{xx}Z +  (A+a)Z   = \Phi 1_{\omega} -ay & \textrm{in}\  Q_T,\\
& Z=0 & \textrm{on}\  \Sigma_T, \\
& (Z(\cdot,0), \partial_t z(\cdot,0))=(0,0) &\textrm{in}\  \Omega,
\end{aligned}
\right.
\end{equation}
and $\Phi:=(\ph_a-\ph)$ solves
\begin{equation}
 \nonumber
\left\{
\begin{aligned}
& \partial_{tt}\Phi - \partial_{xx}\Phi +  (A+a)\Phi  = -a\ph & \textrm{in}\  Q_T,\\
& \Phi=0 & \textrm{on}\  \Sigma_T, \\
& (\Phi(\cdot,0), \partial_t\Phi(\cdot,0))=(\ph_{a,0}-\ph_0,\ph_{a,1}-\ph_1) &\textrm{in}\  \Omega.
\end{aligned}
\right.
\end{equation}
We decompose $\Phi=\Psi+\psi$ where $\Psi$ and $\psi$ solve respectively 
\begin{equation}
 \nonumber
\left\{
\begin{aligned}
& \partial_{tt}\Psi - \partial_{xx}\Psi +  (A+a)\Psi  = 0 & \textrm{in}\  Q_T,\\
& \Psi=0 & \textrm{on}\  \Sigma_T, \\
& (\Psi(\cdot,0), \partial_t\Psi(\cdot,0))=(\ph_{a,0}-\ph_0,\ph_{a,1}-\ph_1) &\textrm{in}\  \Omega,
\end{aligned}
\right.
\qquad \left\{
\begin{aligned}
& \partial_{tt}\psi - \partial_{xx}\psi +  (A+a)\psi   =  -a\ph & \textrm{in}\  Q_T,\\
& \psi=0 & \textrm{on}\  \Sigma_T, \\
& \psi(\cdot,0), \partial_t\psi(\cdot,0))=(0,0) &\textrm{in}\  \Omega,
\end{aligned}
\right.
\end{equation}
and we deduce that $\Psi1_\omega$ is the control of minimal $L^2$ norm for $Z$ solution of  
\begin{equation}
 \nonumber
\left\{
\begin{aligned}
& \partial_{tt}Z - \partial_{xx}Z +  (A+a)Z   = \Psi 1_{\omega} + \Big(\psi 1_\omega-ay\Big) & \textrm{in}\  Q_T,\\
& Z=0 & \textrm{on}\  \Sigma_T, \\
& (Z(\cdot,0), \partial_t Z(\cdot,0))=(0,0) &\textrm{in}\  \Omega.
\end{aligned}
\right.
\end{equation}
Lemma \ref{controllability_result} implies that 

\begin{equation}\nonumber
\Vert \Psi\Vert_{2,q_T}+ \Vert (Z,\partial_t Z)\Vert_{L^\infty(0,T;\boldsymbol{V})}\leq C \Vert \psi 1_\omega-ay\Vert_2 e^{(1+2C)\sqrt{\Vert A\Vert_\infty}}.
\end{equation}
Moreover, energy estimates applied to $\psi$ give
$
\Vert \psi\Vert_{L^2(q_T)}\leq C\Vert a\Vert_\infty \Vert\ph\Vert _{2}e^{C\sqrt{\Vert A+a\Vert_\infty}}
$
and  
$$
\Vert\ph\Vert _2\leq C \Vert \ph_0,\ph_1\Vert_{\boldsymbol{H}}e^{(1+C)\sqrt{\Vert A\Vert_\infty}}\leq \Big(C e^{(1+C)\sqrt{\Vert A\Vert_\infty}}\Big)^2 \Vert u\Vert_{2,q_T}
$$
using that $\Vert \ph_0,\ph_1\Vert_{\boldsymbol{H}}\leq C e^{C\sqrt{\Vert A\Vert_\infty}}\Vert u\Vert_{2,q_T}$
so that 
$$
\Vert \psi\Vert_{L^2(q_T)}\leq C\Vert a\Vert_\infty e^{C\sqrt{\Vert A+a\Vert_\infty}} \Big(C e^{(1+C)\sqrt{\Vert A\Vert_\infty}}\Big)^2 \Vert u\Vert_{2,q_T}
$$
from which we deduce that
$$
\begin{aligned}
\Vert Z\Vert_{L^\infty(Q_T)} & \leq C \Big(\Vert \psi 1_\omega\Vert + \Vert a\Vert_{L^\infty(Q_T)}\Vert y\Vert_2\Big) e^{(1+2C)\sqrt{\Vert A\Vert_\infty}}\\
& \leq C \Vert a \Vert_\infty \Big(e^{C\sqrt{\Vert A+a\Vert_\infty}} \Big(C e^{(1+C)\sqrt{\Vert A\Vert_\infty}}\Big)^2 \Vert u\Vert_{L^2(q_T)}+\Vert y\Vert_{L^2(Q_T)}\Big)\\
& \leq C \Vert a \Vert_\infty \Big(e^{C\sqrt{\Vert A+a\Vert_\infty}} \Big(C e^{(1+C)\sqrt{\Vert A\Vert_\infty}}\Big)^2+1\Big)\\
&\hspace{2cm}+C \Big(\Vert B\Vert_2 e^{(1+C)\sqrt{\Vert A\Vert_\infty}} + \Vert u_0,u_1\Vert_{\boldsymbol{V}}\Big) e^{C\sqrt{\Vert A\Vert_\infty}}
\end{aligned}
$$
leading to the result. 
\end{proof}

This result allows to establish the following property for the operator $K$.

\begin{lemma}\label{lem_final}
Under the assumptions done in Theorem \ref{thm_Zuazua}, let $M=M(\Vert u_0, u_1\Vert_{\boldsymbol{V}},\beta)$ be such that $K$ maps $B_\infty(0,M)$ into itself and assume that $\hat{g}^{\prime}\in L^\infty(0,M)$.
For any $\xi^i\in B_\infty(0,M)$, $i=1,2$, there exists $c(M)>0$ such that  
\begin{equation}
 \nonumber
\Vert K(\xi^2)-K(\xi^1)\Vert_{\infty}\leq c(M) \Vert \hat{g}^{\prime}\Vert_{L^\infty(0,M)} \Vert \xi^2-\xi^1\Vert_\infty.
\end{equation}
\end{lemma}

\begin{proof} For any $\xi^i\in B_\infty(0,M)$, $i=1,2$, let $y_{\xi^i}=K(\xi^i)$ be the null controlled solution of 
\begin{equation}
 \nonumber
\left\{
\begin{aligned}
& \partial_{tt}y_{\xi^i} - \partial_{xx}y_{\xi^i} +  y_{\xi^i} \,\widehat{g}(\xi^i)= -g(0)+f_{\xi^i} 1_{\omega} &\textrm{in}\  Q_T,\\
& y_\xi^i=0 &\textrm{on}\  \Sigma_T, \\
& (y_{\xi^i}(\cdot,0),\partial_ty_{\xi^i}(\cdot,0))=(u_0,u_1) &\textrm{in}\  \Omega,
\end{aligned}
\right.
\end{equation}
with the control $f_{\xi^i} 1_{\omega}$ of minimal $L^2(q_T)$ norm. We observe that $y_{\xi^2}$ is solution of 
\begin{equation}
 \nonumber
\left\{
\begin{aligned}
& \partial_{tt}y_{\xi^2} - \partial_{xx}y_{\xi^2} +  y_{\xi^2} \,\widehat{g}(\xi^1)+ y_{\xi^2}(\widehat{g}(\xi^2)-\widehat{g}(\xi^1))= -g(0)+f_{\xi^2} 1_{\omega} &\textrm{in}\  Q_T,\\
& y_{\xi^2}=0 &\textrm{on}\  \Sigma_T, \\
& (y_{\xi^2}(\cdot,0),\partial_t y_{\xi^2}(\cdot,0))=(u_0,u_1) &\textrm{in}\  \Omega.
\end{aligned}
\right.
\end{equation}
It follows from Lemma \ref{gap} applied with $B=-g(0)$, $A=\hat{g}(\xi^1)$, $a=\hat{g}(\xi^2)-\hat{g}(\xi^1)$, that 
\begin{equation}\label{gapxi}
\Vert y_{\xi^2}-y_{\xi^1} \Vert_{\infty}\leq A(\xi^1,\xi^2)\Vert \widehat{g}(\xi^2)-\widehat{g}(\xi^1)\Vert_\infty
\end{equation}
where the positive constant  
\begin{multline*}
A(\xi^1,\xi^2):= C^2 \Big(e^{C\sqrt{\Vert  \hat{g}(\xi^2)\Vert_\infty}} \Big(C e^{(1+C)\sqrt{\Vert  \hat{g}(\xi^1)\Vert_\infty}}\Big)^2\Big) \\
\Big(\Vert g(0)\Vert_2 e^{(1+C)\sqrt{\Vert  \hat{g}(\xi^1)\Vert_\infty}} + \Vert u_0,u_1\Vert_{\boldsymbol{V}}\Big) e^{C\sqrt{\Vert  \hat{g}(\xi^1)\Vert_\infty}}
\end{multline*}
is bounded by some $c(M)>0$ for every $\xi^i\in B_\infty(0,M)$. The result follows from \eqref{gapxi}. 
\end{proof}

\begin{remark}\label{rem_contract}
By Lemma \ref{lem_final}, if $\Vert \hat{g}^{\prime}\Vert_{L^\infty(0,M)}<  1/c(M)$ then the operator $K$ is contracting. Note however that the bound depends on the norm $\Vert u_0,u_1\Vert_{\boldsymbol{V}}$ of the initial data to be controlled.
\end{remark}

\medskip
\noindent\textbf{Acknowlegment.}
The first author warmly thanks J\'erome Lemoine (Laboratoire Math\'ematique Blaise Pascal, Clermont Auvergne University) for fruitful discussions.

{\small
\bibliographystyle{siam}
\bibliography{wavecontinuation.bib}

\begin{thebibliography}{10}

\bibitem{Cannarsa_Loreti_Komornik2002}
{\sc P.~Cannarsa, V.~Komornik, and P.~Loreti}, {\em One-sided and internal
  controllability of semilinear wave equations with infinitely iterated
  logarithms}, Discrete Contin. Dyn. Syst., 8 (2002), pp.~745--756.

\bibitem{CazenaveHaraux1980}
{\sc T.~Cazenave and A.~Haraux}, {\em \'{E}quations d'\'{e}volution avec non
  lin\'{e}arit\'{e} logarithmique}, Ann. Fac. Sci. Toulouse Math. (5), 2
  (1980), pp.~21--51.

\bibitem{cindea_efc_munch_2013}
{\sc N.~C\^{\i}ndea, E.~Fern\'{a}ndez-Cara, and A.~M\"{u}nch}, {\em Numerical
  controllability of the wave equation through primal methods and {C}arleman
  estimates}, ESAIM Control Optim. Calc. Var., 19 (2013), pp.~1076--1108.

\bibitem{Coron-Book-07}
{\sc J.-M. Coron}, {\em Control and nonlinearity}, vol.~136 of Mathematical
  Surveys and Monographs, American Mathematical Society, Providence, RI, 2007.

\bibitem{coron-trelat-wave-06}
{\sc J.-M. Coron and E.~Tr\'{e}lat}, {\em Global steady-state stabilization and
  controllability of 1{D} semilinear wave equations}, Commun. Contemp. Math., 8
  (2006), pp.~535--567.

\bibitem{deuflhard}
{\sc P.~Deuflhard}, {\em Newton methods for nonlinear problems}, vol.~35 of
  Springer Series in Computational Mathematics, Springer-Verlag, Berlin, 2004.
\newblock Affine invariance and adaptive algorithms.

\bibitem{EFC-AM-2012}
{\sc E.~Fern\'{a}ndez-Cara and A.~M\"{u}nch}, {\em Numerical null
  controllability of semi-linear 1-{D} heat equations: fixed point, least
  squares and {N}ewton methods}, Math. Control Relat. Fields, 2 (2012),
  pp.~217--246.

\bibitem{lemoine_gayte_munch}
{\sc J.~Lemoine, I.~Gayte, and A.~M{\"u}nch}, {\em Approximation of nulls
  controls for semilinear heat equations using a least-squares approach},
  Preprint. arXiv:2008.12656.

\bibitem{lemoinemunch_time}
{\sc J.~Lemoine and A.~M{\"u}nch}, {\em A fully space-time least-squares method
  for the unsteady {N}avier-{S}tokes system}, Preprint. arXiv:1909.05034.

\bibitem{lemoine-Munch-Pedregal-AMO-20}
{\sc J.~Lemoine, A.~M\"{u}nch, and P.~Pedregal}, {\em Analysis of continuous
  ${H}^{-1}$-least-squares approaches for the steady {N}avier-{S}tokes system},
  To appear in Applied Mathematics and Optimization,  (2020).

\bibitem{Li_Zhang_2000}
{\sc L.~Li and X.~Zhang}, {\em Exact controllability for semilinear wave
  equations}, J. Math. Anal. Appl., 250 (2000), pp.~589--597.

\bibitem{JLL88}
{\sc J.-L. Lions}, {\em Contr\^olabilit\'e exacte, perturbations et
  stabilisation de syst\`emes distribu\'es. {T}ome 1}, vol.~8 of Recherches en
  Math\'ematiques Appliqu\'ees [Research in Applied Mathematics], Masson,
  Paris, 1988.
\newblock Contr{\^o}labilit{\'e} exacte. [Exact controllability], With
  appendices by E. Zuazua, C. Bardos, G. Lebeau and J. Rauch.

\bibitem{AM-PP-2014}
{\sc A.~M\"{u}nch and P.~Pedregal}, {\em Numerical null controllability of the
  heat equation through a least squares and variational approach}, European J.
  Appl. Math., 25 (2014), pp.~277--306.

\bibitem{saramito}
{\sc P.~Saramito}, {\em A damped {N}ewton algorithm for computing viscoplastic
  fluid flows}, J. Non-Newton. Fluid Mech., 238 (2016), pp.~6--15.

\bibitem{Zhang_2000}
{\sc X.~Zhang}, {\em Explicit observability estimate for the wave equation with
  potential and its application}, R. Soc. Lond. Proc. Ser. A Math. Phys. Eng.
  Sci., 456 (2000), pp.~1101--1115.

\bibitem{zuazua90}
{\sc E.~Zuazua}, {\em Exact controllability for the semilinear wave equation},
  J. Math. Pures Appl. (9), 69 (1990), pp.~1--31.

\bibitem{zuazua91}
{\sc E.~Zuazua}, {\em Exact boundary controllability for the semilinear wave
  equation}, in Nonlinear partial differential equations and their
  applications. {C}oll\`ege de {F}rance {S}eminar, {V}ol.\ {X} ({P}aris,
  1987--1988), vol.~220 of Pitman Res. Notes Math. Ser., Longman Sci. Tech.,
  Harlow, 1991, pp.~357--391.

\bibitem{zuazua93}
\leavevmode\vrule height 2pt depth -1.6pt width 23pt, {\em Exact
  controllability for semilinear wave equations in one space dimension}, Ann.
  Inst. H. Poincar\'e Anal. Non Lin\'eaire, 10 (1993), pp.~109--129.

\end{thebibliography}
}

\end{document}